\documentclass[12pt]{article}
\usepackage{amssymb,amscd}
\usepackage{amsmath, amsthm}
\usepackage{graphicx}
\usepackage{psfrag}
\usepackage{comment}

\begin{document}

\newtheorem{lemma}{Lemma}[section]
\newtheorem{corollary}{Corollary}[section]
\newtheorem{theorem}{Theorem}[section]
\newtheorem{definition}{Definition}[section]
\newtheorem{note}{Note}[section]

\begin{center}
{\bf \Large The Critical Locus for Complex H\'{e}non Maps.}\medskip

Tanya Firsova
\end{center}

\begin{abstract} We give a topological model of the critical locus for
complex H\'{e}non maps that are perturbations of the quadratic
polynomial with disconnected Julia set.
\end{abstract}

\def\IMSmarkvadjust{0 pt}
\def\IMSmarkhadjust{0 pt}
\def\IMSmarkhpadding{0 pt}
\def\IMSpubltext{Published in modified form:}
\def\SBIMSMark#1#2#3{
 \font\SBF=cmss10 at 10 true pt
 \font\SBI=cmssi10 at 10 true pt
 \setbox0=\hbox{\SBF \hbox to \IMSmarkhpadding{\relax}
                Stony Brook IMS Preprint \##1}
 \setbox2=\hbox to \wd0{\hfil \SBI #2}
 \setbox4=\hbox to \wd0{\hfil \SBI #3}
 \setbox6=\hbox to \wd0{\hss
             \vbox{\hsize=\wd0 \parskip=0pt \baselineskip=10 true pt
                   \copy0 \break%
                   \copy2 \break%
                   \copy4 \break}}
 \dimen0=\ht6   \advance\dimen0 by \vsize \advance\dimen0 by 8 true pt
                \advance\dimen0 by -\pagetotal
	        \advance\dimen0 by \IMSmarkvadjust
 \dimen2=\hsize \advance\dimen2 by .25 true in
	        \advance\dimen2 by \IMSmarkhadjust

%
%
  \openin2=publishd.tex
  \ifeof2\setbox0=\hbox to 0pt{}
  \else 
     \setbox0=\hbox to 3.1 true in{
                \vbox to \ht6{\hsize=3 true in \parskip=0pt  \noindent  
                {\SBI \IMSpubltext}\hfil\break
                \input publishd.tex 
                \vfill}}
  \fi
  \closein2
  \ht0=0pt \dp0=0pt
 \ht6=0pt \dp6=0pt
 \setbox8=\vbox to \dimen0{\vfill \hbox to \dimen2{\copy0 \hss \copy6}}
 \ht8=0pt \dp8=0pt \wd8=0pt
 \copy8
 \message{*** Stony Brook IMS Preprint #1, #2. #3 ***}
}

\SBIMSMark{2011/1}{February 2011}{}

\section{Introduction. Foliations.}\label{introduction}

The family of H\'{e}non mappings is a basic example of nonlinear
dynamics. Both real and complex versions of these maps were
extensively studied, but still there is a great deal of unknown
about them.

In this article we study complex quadratic H\'{e}non mappings.
These are maps of the form:

$$f_a(x)=\left(\begin{array}{c} x^2+c-ay \\ x \end{array}\right).$$


\noindent Note that these maps are biholomorphisms of $\mathbb
C^2$. They have constant Jacobian, which equals to the parameter
$a$. As $a\to 0$, H\'{e}non mappings degenerate to quadratic
polynomial maps $x\mapsto x^2+c$, which act on the parabola
$x=y^2+c$. The H\'{e}non mappings that we study are perturbations
of quadratic polynomials with disconnected Julia set.

In analogy with one-dimensional dynamics, Hubbard and
Oberste-Vorth \cite{HOV1} introduced the functions $G_a^+$ and
$G_a^-$ that measure the growth rate of the forward and backwards
iterations of the orbit. These functions are pluriharmonic on the
set of points $U_a^+$, $U_a^-$, whose orbits tend to infinity
under forward and backwards iterates. Their level sets are
foliated by Riemann surfaces. These natural foliations ${\cal
F}_a^+$ and ${\cal F}_a^-$ were introduced and extensively studied
in \cite{HOV1}.






The foliations ${\cal F}_a^+, {\cal F}_a^-$ can also be
characterized in terms of B\"{o}ttcher coordinates. There are maps
$\phi_{a,+}$ and $\phi_{a,-}$ that semiconjugate the map $f_a$ and
the map $f_a^{-1}$ to $z\mapsto z^2$ and $z\mapsto z^2/a$
correspondingly. These functions were constructed in \cite{HOV1}.
We recall the definitions of these functions and list their
properties in Section \ref{bottcher_coordinates}. The level sets
of $\phi_{a,+}$ and $\phi_{a,-}$ define ${\cal F}_a^+$ and ${\cal
F}_a^-$ in some domains near the infinity and can be propagated by
dynamics to all of $U_a^+$ and $U_a^-$ correspondingly.

The dynamical description of these foliations is the following:






\begin{lemma}[\cite{BS4}] The leaves of ${\cal F}_a^+$ are
``super-stable manifolds of $\infty$"\ , i.e. if $z_1,z_2$ belong
to the same leaf, then $d(f_a^n(z_1), f_a^n(z_2))\to 0$
super-exponentially, where $d$ is Euclidean distance in $\mathbb
C^2$. If $z_1,z_2\in U_a^+$ do not belong to the same leaf, then
$d(f_a^n(z_1), f_a^n(z_2))\not \to 0$.
\end{lemma}

\noindent The leaves of ${\cal F}_a^-$ are ``super-unstable
manifolds of $\infty$".

One would like to think of these foliations as coordinates in
$U_a^+\cap U_a^-$. However, for all H\'{e}non mappings there is a
codimension one subvariety of tangencies between ${\cal F}_a^+$
and ${\cal F}_a^-$\cite{BS4}.

\begin{definition} The critical locus ${\cal C}_a$ is the set of tangencies
between foliations ${\cal F}_a^+$ and ${\cal F}_a^-$.
\end{definition}

\noindent Thus, the critical locus is the set of ``heteroclinic
tangencies"\ between the ``super-stable"\ and ``super-unstable"\
manifolds.

We give an explicit description of the critical locus for
H\'{e}non mappings $$(x,y)\mapsto(x^2+c-ay,x),$$ \noindent where
$x^2+c$ has disconnected Julia set, and $a$ is sufficiently small.
The topological model of the critical locus for such H\'{e}non
maps was conjectured by John Hubbard. We justify his picture.

Lyubich and Robertson (\cite{LR}) gave the description of the
critical locus for H\'{e}non mappings $$(x,y)\mapsto
(p(x)-ay,x),$$ \noindent where $p(x)$ is a hyperbolic polynomial
with connected Julia set, $a$ is sufficiently small. They showed
that for each critical point $c$ of $p$ there is a component of
the critical locus that is asymptotic to the line $y=c$. Each
component of the critical locus is an iterate of these ones, and
each is a punctured disk.

They used critical locus to show that a pair of quadratic
H\'{e}non maps of the studied type, taken along with the natural
foliations, gives a rigid object. This means that if a conjugacy
between two H\'{e}non maps sends the natural foliations of the
first map to the natural foliations of the second map then the two
H\'{e}non maps are conjugated by a holomorphic or antiholomorphic
affine map.

Since the work of Lyubich and Robertson \cite{LR} is unpublished,
we include the proof of all the results that we are using.

Some of the sources of the fundamental results about H\'{e}non
mappings are \cite{BLS}, \cite{BS1}, \cite{BS2}, \cite{BS3},
\cite{BS4}, \cite{BS5}, \cite{BS6}, \cite{HOV1}, \cite{HOVII},
\cite{FM}, \cite{FS}.

\subsection{Acknowledgements.} The author is grateful to Mikhail
Lyubich for statement of the problem, numerous discussions and
useful suggestions. We are thankful to John Hubbard for showing
the conjectural picture of the critical locus.

\section{Topological model for the critical locus.}

In our case truncated spheres serve as building blocks for the
critical locus. Consider a sphere $S$ and a pair of disjoint
Cantor sets $\Sigma, \Omega \subset S$. The elements of $\Sigma$,
$\Omega$ can be parametrized by one-sided infinite sequences of
$0,1$'s. Denote by $\sigma_{\alpha}\in \Sigma$,
$\omega_{\alpha}\in\Omega$ elements parametrized by a sequence
$\alpha$.

Let $\alpha_n$ be a $n$-string of $0$'s and $1$'s. For each
$\alpha_n$, $n\in \mathbb N\cup \{0\}$, take a disk
$V_{\alpha_n}\subset S\backslash\left(\Sigma\cup \Omega\right)$
with the smooth boundary. (Let $V$ denote the disk that
corresponds to an empty sequence.) We require that $V_{\alpha_n}$
are disjoint. Moreover, $V_{\alpha_n}$ converge to
$\sigma_{\alpha}$, where $\alpha_n$ is the string of the first $n$
elements of $\alpha$.

\begin{figure}[h]
\centering
\psfrag{sigma}{$\Sigma$}
\includegraphics[height=3.5cm]{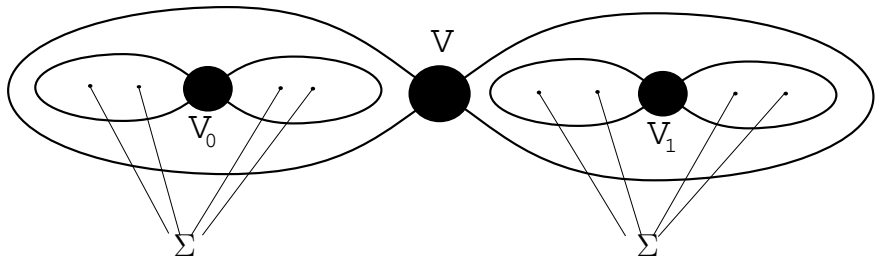}
\caption{The geometry of a truncated sphere}
\end{figure}

$U_{\alpha_n}$ play the same role for $\Omega$ as $V_{\alpha_n}$
for $\Sigma$.

We assume that there is a fixed homeomorphism $\tilde{h}$ between
the boundaries of $V_{\alpha_n}$ and $U_{\alpha_n}$.

Let $p\in S$ be a point.

We say that $S\backslash\left(\Sigma\cup\Omega\cup\sum_{\alpha_n}
[U_{\alpha_n}\cup V_{\alpha_n}]\cup p\right)$ is a truncated
sphere.

Note that all truncated spheres are homeomorphic one to another.

\begin{theorem}\label{theorem_main} Suppose $x^2+c$ has disconnected Julia set. There
exists $\delta$, such that $\forall |a|<\delta$ the critical locus
of the map
$$
f_a\left(\begin{array}{c}x \\y \end{array}\right)=
\left(\begin{array}{c}x^2+c-ay \\ x \end{array}\right)
$$
is smooth, and has the following topological model: take
countably many truncated spheres $S_n$, $n\in \mathbb Z$, and glue
the boundary of $V_{\alpha_n}$ on $S_k$ to the boundary of
$U_{\alpha_n}$ of $S_{n+k}$ using the homeomorphism $\tilde{h}$.
Map $f_a$ acts on the critical locus by sending $S_n$ to
$S_{n+1}$.
\end{theorem}

\begin{figure}[h]
\centering
\includegraphics[height=11cm]{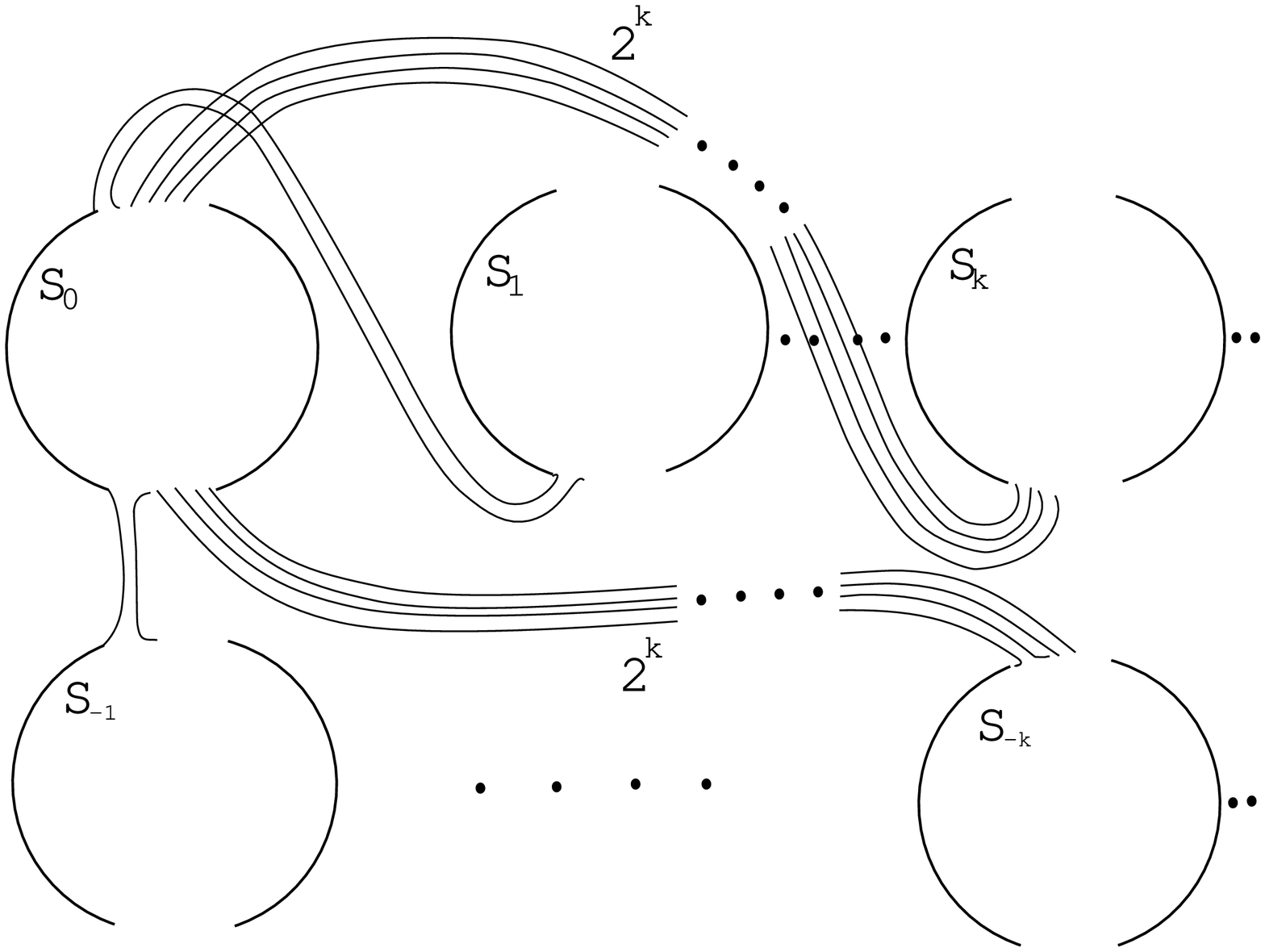}
\caption{The topological model.}
\end{figure}

\begin{note} It is easy to see that the topological model of the critical locus is
well-defined up to a homeomorphism.
\end{note}

\section{Strategy of description.}\label{strategy}

When $a=0$, the H\'{e}non mapping reduces to

$$\left(\begin{array}{c}x \\ y\end{array}\right) \mapsto \left(\begin{array}{c} x^2+c \\ x\end{array}\right)$$

\noindent which is a map $x\mapsto p(x)$ acting on the curve
$x=p(y)$.

As $a\to 0$ the map degenerates, but the foliations and the Green
functions persist and become easy to analyze.

In Section \ref{bottcher_coordinates} we study  $\phi_{a,+}$ and
$\phi_{a,-}$, paying extra attention to the degeneration as $a\to
0$.

Section \ref{sec:Green_function} is devoted to Green's functions.
We prove that $G_a^+$ and $G_a^-$ depend continuously on $x,y$ and
the parameter $a$.

In Section \ref{sec:critical_locus} we describe the foliations
${\cal F}_a^+$ and ${\cal F}_a^-$ and the critical locus $\cal C$
in terms of $\phi_{a,+}$ and $\phi_{a,-}$. We also calculate the
critical locus in the degenerate case. As $a$ deviates from zero,
we carefully describe the perturbation. The latter is done in
several steps:

First, we choose appropriate values $r$ and $\alpha$ and describe
the critical locus in the domain

$$\Omega_a = \{G_a^+\leq r\}\cap
\{|y|\leq \alpha\}\cap \{|p(y)-x|\geq |a|\alpha\}.$$

\begin{figure}[h]
\centering
\psfrag{omega}{$\Omega_a$} \psfrag{G=r}{$G_a^+=r$}
\psfrag{|a|alpha}{$|a|\alpha$}

\psfrag{|y|=alpha}{$|y|=\alpha$}

\includegraphics[height=6cm]{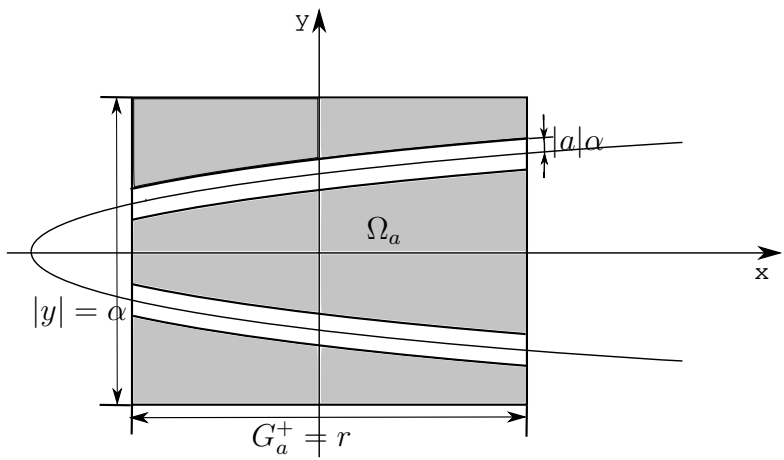}
\caption{Domain $\Omega_a$}\label{fig:Omega}
\end{figure}

We choose $r$ that lies between the critical point level and
critical value level:

$$G_p(0)<r<G_p(c),$$

\noindent where $G_p$ is the Green function of polynomial $p$.

We choose $\alpha$ such that $G_a^+|_{\{|x|>\alpha, |x|>|y|\}}>r$.
Moreover, in Section \ref{F-} we choose $\alpha$ such that leaves
of foliation ${\cal F}_a^-$ in $\Omega_a$ form a family horizontal
of parabolas.

\begin{figure}[h!]
\centering
\psfrag{y}{$y$}\psfrag{x}{$x$}\psfrag{x=xim}{$x=p^{-m}(0)$}\psfrag{x=xi1}{$x=p^{-1}(0)$}\psfrag{x=xik}{$x=p^{-k}(0)$}
\psfrag{J+}{$J_a^+$}
\includegraphics[height=6cm]{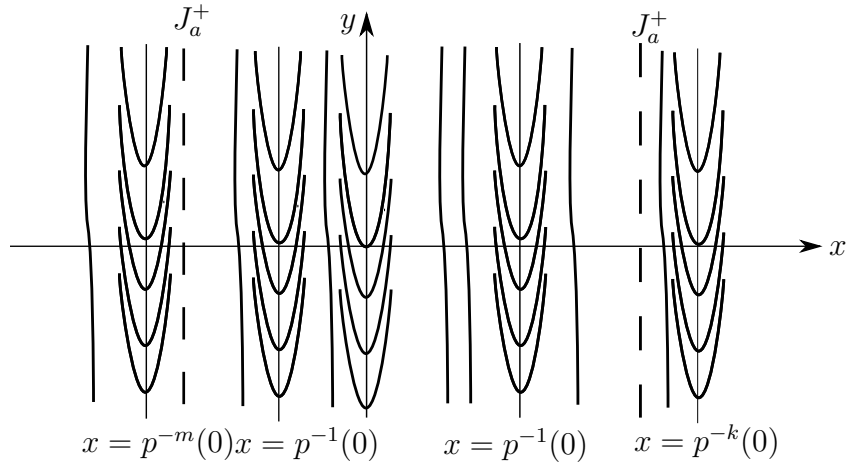}
\caption{The foliation ${\cal F}_a^+$}\label{fig:F+}
\end{figure}

In Section \ref{F+} we give a description of the foliation ${\cal
F}_a^+$ in

$$\{G_a^+\leq r\}\cap \{|y|\leq \alpha\}.$$

\noindent We show that the local leaves of ${\cal F}_a^+$ are
either vertical-like or vertical parabolas. There is exactly one
family of vertical parabolas (see Fig. \ref{fig:F+}) corresponding
to each line $x=\xi_k,$ where $p^{\circ k}(\xi_k)=0$.

The critical locus in $\Omega_a$ is the set of tangencies of a
family of horizontal parabolas with vertical-like leaves and
vertical parabolic leaves. In Section \ref{Omega} we show that for
each family of parabolas there is exactly one handle.

In Section \ref{infinity} we show that foliations ${\cal F}_a^+$
and ${\cal F}_a^-$ extend holomorphically to $x=\infty$. We show
that a point $(0,\infty)$ belongs to the extension and calculate
the tangent line to the critical locus at this point. This gives
us a description of the critical locus in
$$\{|y|<\epsilon\}\cap \{|x|>\alpha\}.$$

In Section \ref{ext-a-neighborhood} we show that the critical
locus can be extended up to $a|\alpha|$-neighborhood of parabola
$p(y)-x=\mbox{const}$ along $y=0$.

In Section \ref{final} we combine the results from the previous
sections to get the description of the fundamental domain of the
certain component of the critical locus. We rule out ghost
components. We also do a dynamical regluing of the fundamental
domain of the critical locus to obtain a description in terms of
truncated spheres.

\section{Functions $\phi_{a,+}$ and
$\phi_{a,-}$}\label{bottcher_coordinates}

In this section we construct functions $\phi_{a,+}$ and
$\phi_{a,-}$. In their description we follow \cite{HOV1}. In
description of the degeneration of $\phi_{a,+}$ and $\phi_{a,-}$
as $a\to 0$, we follow \cite{LR}.

\subsection{Large scale behavior of the H\'{e}non map}

The study of H\'{e}non mappings usually begins with introduction
of domains $V_{+}$ and $V_{-}$ which are invariant under the
action of H\'{e}non map:
$$f_{a}(V_{+})\subset V_{+}, \quad f^{-1}_{a}(V_{-})\subset V_{-}.$$

Moreover, one requires that every point that has unbounded forward
orbit eventually enters $V_{+}$ and every point that has unbounded
backward orbit enters $V_{-}$.

Fix $R$, $0<r<1$, and choose $\alpha$ so that

\begin{equation}\label{1}
\frac{|c|}{|y^2|}+\frac{R+1}{|y|}<r
\tag{\ref{bottcher_coordinates}.1}
\end{equation}

\begin{equation}\label{2}
|p(y)|>(2R+1)|y|\tag{\ref{bottcher_coordinates}.2}
\end{equation}

\noindent for all $|y|\geq \alpha$.

Consider the following partition of $\mathbb C^2:$

$$V_{+}=\{(x,y) :\ |x|> |y|, |x|> |\alpha|\};$$

$$V_{-}=\{(x,y) :\ |y|> |x|, |y|>\alpha\};$$

$$W=\{(x,y) :\ |x|\leq \alpha, |y|\leq \alpha \}.$$

\begin{figure}[h]
\centering
\psfrag{ex1}{$V_-$} \psfrag{ex2}{$V_+$} \psfrag{R}{W}

\includegraphics[height=5cm]{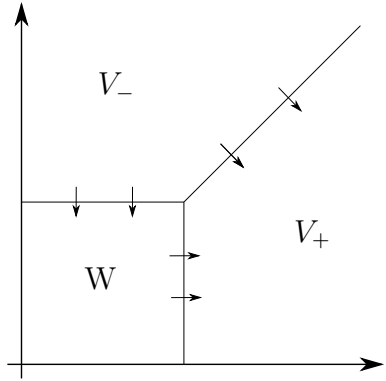}
\caption{The crude picture of the dynamics of $f_a$}
\end{figure}

\begin{lemma}\label{V+invariance} For $a\in D_R$, $f_a(V_+)\subset V_+$.
\end{lemma}

\begin{proof}

$|p(x)-ay|\geq |p(x)|+|a||y|\geq (2R+1-R)|y|\geq |y|$

$|p(x)-ay|\geq (R+1)|y|\geq \alpha$
\end{proof}

Let

$$\left( \begin{array}{l}x_n\\y_n\end{array}\right)=f^n_a\left( \begin{array}{l}x\\y\end{array}\right).$$

\begin{lemma}
$U_a^+=\bigcup_{n}f^{-n}_a(V_{+});$
\end{lemma}

\begin{proof}
Let $f_a^n(x)\to \infty$. For some $n$, $f_a^n(x)\in V_{+}$ or
$f_a^n(x)\in V_{-}$.

Suppose that for all $n\geq n_0$ $(x_n,y_n)\in V_-$. Note that
$|y_{n+1}|=|x_n|\leq |y_{n}|$. Sequence $\{|y_n|\}$ is decreasing.
Contradiction.
\end{proof}

\begin{lemma} For $a\in D_R$ $f^{-1}_{a}(V_-)\subset V_{-}$.
\end{lemma}

\begin{proof}

$\frac{|p(y)-x|}{|a|}\geq \frac{|2R+1||y|-|x|}{R}\geq |y|$

$\frac{|p(y)-x|}{|a|}\geq |y|\geq |\alpha|$
\end{proof}

\begin{lemma}
$U_a^-=\cup_{n}f^n_a(V_{-}).$
\end{lemma}

\begin{proof} The proof is the same as in the previous lemma.
\end{proof}

\subsection{Function $\phi_{a,+}$}

The function $\phi_{a,+}$ is constructed as the limit

$$\phi_{a,+}(x,y)=\lim_{n\to \infty} |x_n|^{\frac{1}{2^n}}.$$

\noindent with the appropriate choice of the branch of root. We
are particularly interested in this function in $V_+$. Sense is
made of the above limit using the telescopic formula

\begin{equation}\label{telescopic_formula1}
\phi_{a,+}=\lim_{n\to \infty}x \exp
\left(\frac{1}{2}\log\frac{x_1}{x^2}+\dots
+\frac{1}{2^n}\log\frac{x_n}{x_{n-1}^2}+\dots\right).
\end{equation}

\begin{lemma} The function $\phi_{a,+}$, defined by
formula (\ref{telescopic_formula1}), is well-defined and
holomorphic for all $(x,y)\in V_+$ and all $a\in D_R$. Moreover,
there exists $B$ so that

$$B^{-1}<\left|\frac{\phi_{a,+}}{x}\right|<B.$$
\end{lemma}

\begin{proof}

Let $s_n^+=\frac{x_n}{x^2_{n-1}}-1$. By property (\ref{1}),
$|s_n^+|<r$.

\noindent$\log\frac{x_n}{x_{n-1}^2}=\log(1+s_n^+)$ is calculated
using the principle branch of $\log$. Therefore, the series

\begin{equation}\label{sum1}
\frac{1}{2}\log\frac{x_1}{x^2}+\dots+\frac{1}{2^n}\log\frac{x_n}{x_{n-1}^2}+\dots
\end{equation}

\noindent converges absolutely and uniformly. Since
$|\log\frac{x_n}{x_{n-1}^n}|<-\frac{\log(1-r)}{2^n}$, the infinite
sum (\ref{sum1}) is no bigger than $-\log B= -\log(1-r)$.

The final claim follows immediately from the expression
(\ref{telescopic_formula1}) and the bound derived for the series
(\ref{sum1}).
\end{proof}

We show that $\phi_{a,+}\sim x$ as $x\to \infty$ in Section
\ref{infinity}.

Let ${\cal D}_{n,+}=\{(x,y,a)|\ f^n_a(x,y,a)\in V_+, a\in D_R\}$

\begin{lemma} The function $\phi^{2^n}_{a,+}$ extends to a
holomorphic function on ${\cal D}_{n,+}$ given by
$\phi^{2^n}_{a,+}=\phi_{a,+}\circ f^{n}_a$. Moreover,
$\phi^{2^n}_{0,+}(x,y)=b^{2^n}_p(x)$.
\end{lemma}

\begin{proof} The function $\phi^{2^n}_{a,+}$ is holomorphic by
definition.

As $a\to 0$ the H\'{e}non mappings degenerate to a one-dimensional
map $x\to p(x)$, acting on $y=p(x)$. Therefore,
$\phi_{0,+}(x,y)=b_p(x)$ on $V_+$, and
$\phi_{0,+}^{2^n}(x,y)=b_p^{2^n}(x)$ if $f_0^n(x,y)\in V_+$.
\end{proof}

Note that 
$$K_0^+=J_0^+=J_p\times \mathbb C,$$
$$U_0^+=U_p\times \mathbb C,$$
where $J_p$ is the Julia set for the one-dimensional map
$x\mapsto p(x)$; $U_p$ is the set of points, whose orbits escape
to $\infty$ under the map $x\mapsto p(x)$.

\subsection{Function $\phi_{a,-}$}

We start by working the leading terms of $y_{-n}$ as a polynomial
in $y$ and as a polynomial in $\frac{1}{a}$. The following
notation simplifies the statement of the result:

$\sigma_k=1+2+\dots+2^{k-1}$ for $k\geq 1$ and $\sigma_k=0$ for
$k\leq 0$.

By an easy induction we get:

\begin{lemma}\label{pol_decomposition} The leading term of $y_{-n}(x,y,a)$ considered as a
polynomial in $y$ is $y^{2^n}/a^{\sigma_n}$. The leading term of
$y_{-n}$ considered as a polynomial in $\frac{1}{a}$ is
$\frac{1}{a^{\sigma_n}}(p(y)-x)^{2^{n-1}}$.
\end{lemma}

We define the function $\phi_{a,-}$ as a limit

\begin{equation}\phi_{a,-}=\lim_{n\to \infty}(y_{-n}\circ
a^{\sigma_n})^{\frac{1}{2^n}}\end{equation}

\noindent with an appropriate choice of branch of root. Note that
the factor $a^{\sigma_n}$ is chosen, so that the leading term of
$y_n\cdot a^{\sigma_n}$, as a polynomial in $y$, is $y^{2^n}$.

Let $D_R^*=D_R\backslash\{0\}$.

Sense is made of the limit in $V_-\times D^*_R$ using the
telescopic formula:

\begin{equation}\label{telescopic_formula2}
\phi_{a,-}(x,y)=\lim_{n\to
\infty}\exp\left(\frac{1}{2}\log\frac{ay_{-1}}{y^2}+\frac{1}{2^2}\log\frac{ay_{-2}}{y^2_{-1}}+\dots\right)
\end{equation}

\begin{lemma}\label{existence_phi-} The function $\phi_{a,-}$, defined by formula
(\ref{telescopic_formula2}) is well-defined and holomorphic on
$V_-\times D_R^*$. Moreover, there exists $B$ so that
$$B^{-1}<\left|\frac{\phi_{a,-}}{y}\right|<B.$$

\end{lemma}

\begin{proof}
Let $s_n^-=\frac{c-x_{-(n-1)}}{y^2_{-(n-1)}}$

Note that $|s_n^-|<r$ for $(x,y)\in V_-$, $a\in D^*_R$.

We evaluate $\frac{ay_{-n}}{y^2_{-(n-1)}}=\log(1+s_n^-)$ using the
principle branch of $\log$.

Since $\log(1+s_n^-)\leq -\log(1-r)$, the series

\begin{equation}\label{sum2}
\frac{1}{2}\log\frac{ay_{-1}}{y^2}+\frac{1}{2^2}\frac{ay_{-2}}{y^2_{-1}}+\dots
\end{equation}

\noindent converges uniformly and absolutely to a holomorphic
function bounded by $-\log(1-r)$. Letting $B=(r-1)^{-1}$, the last
statement of the Lemma follows.
\end{proof}

We show that $\phi_{a,-}\sim y$ as $y\to \infty$ in Section
\ref{infinity}.

The next lemma states that one can extend $\phi_{a,-}$ to a
holomorphic function on $V_-\times D_R$.

\begin{lemma}\label{limiting_value} $\phi_{a,-}=(p(y)-x)^{\frac{1}{2}}+ah(x,y,a)$ for
some holomorphic function $h$ on $V_-\times D_R$.
\end{lemma}

\begin{proof} By Lemma \ref{pol_decomposition} the leading term of $y_{-n}$
as a polynomial in $\frac{1}{a}$ is
$\frac{1}{a^{\sigma_n}}(p(y)-x)^{2^{n-1}}$. Since
$x_{-n}=y_{-(n-1)}$. The leading term of $x_{-n}$ as a polynomial
in $\frac{1}{a}$ is the leading term of $y_{-(n-1)}$.

Recall $s_n^-=\frac{c-x_{-(n-1)}}{y^2_{-(n-1)}}$. It follows that
$s_n^-$ is a polynomial in $a$ and it vanishes in $a$ to the order
$2\sigma_{n-1}-\sigma_{n-2}$. Hence the series (\ref{sum2}) takes
the form

$$
\frac{1}{2}\log\frac{ay_{-1}}{y^2}+ag(x,y,a),
$$
\noindent where $g(x,y,a)$ is a holomorphic function on $V_+\times
D_R$.

By (\ref{telescopic_formula2}), $\phi_{a,-}(x,y)=y \exp
\left(\frac{1}{2}\log \frac{ay_1}{y^2})\right)\exp(ag(x,y,a))$.
The conclusion follows.
\end{proof}

Let
$$C(p)=\{p(y)-x=0\}$$
Domain $f_a(V_{-})$ swells to $\mathbb C^2\backslash
C(p)$ as $a\to 0$.

\begin{lemma}\label{lem:swell} 
For $(x,y)\in \mathbb C^2\backslash C(p)$ we have $(x,y)\in
f_a(V_{-})$ for all sufficiently small $a$. Moreover, for all
$K\Subset \mathbb C^2\backslash C(p)$ $\Rightarrow$ $K\Subset
f_a(V_{-})$ for all small enough $a$.
\end{lemma}

\begin{proof} Both statements follow from
$$f_a(V_{-})=\{|p(y)-x|\geq |a|\alpha, |p(y)-x|\geq
|a||y|\}.$$
\end{proof}

$${\cal D}_{-,n}=\{(x,y,a)|\ (x,y)\in f_a^{n}(V_-), a\neq
0\ \mbox{or}\ (x,y)\in \mathbb C^2\backslash C(p), a=0 \}.$$

\begin{lemma}\label{lim_value2} The function $\phi_{a,-}^{2^n}(x,y)$
extends from a holomorphic function on $V_{-}\times D_R$ to ${\cal
D}_{-,n}$, by letting

\begin{enumerate}
\item $\phi^{2^n}_{a,-}=a^{\sigma_n}\phi_{a,-}\circ f_a^{-n}$
for $a\neq 0$;
\item $\phi^{2^n}_{0,-}(x,y)=(p(y)-x)^{2^{n-1}}$.
\end{enumerate}
\end{lemma}

\begin{proof} If $a\neq 0$, then we can extend the function
$\phi_{a,-}$ to be holomorphic by defining
$\phi^{2^n}_{a,-}=a^{\sigma_n}\phi_{a,-}\circ f_a^{-n}$. This
agrees with $\phi_{a,-}^{2^n}$ on $V_{-}$.

By Lemma \ref{limiting_value}

\begin{equation}
\phi^{2^n}_{a,-}=a^{\sigma_n}\phi_{a,-}\circ
f_a^{-n}=a^{\sigma_n}(p(y_{-n})-x_{-n})^{\frac{1}{2}}+a^{\sigma_n+1}h(x_{-n},
y_{-n},a)
\end{equation}

By Lemma \ref{pol_decomposition} the leading term of $y_{-n}$, as
a polynomial in $\frac{1}{a}$, is
$\frac{(p(y)-x)^{2^n}}{a^{\sigma}}$. Since $x_{-n}=y_{-(n-1)}$,
the leading term of $x_{-n}$ as a polynomial in $\frac{1}{a}$ is
$\frac{(p(y)-x)^{2^{n-2}}}{a^{\sigma_{n-1}}}$. Thus
$a^{\sigma_n}(p(y_{-n})-x_{-n})^{\frac{1}{2}}$ is a polynomial in
$x,y,a$ and the only term free in $a$ is $(p(y)-x)^{2^{n-1}}$.

By the same argument the function $a^{\sigma_n}h(x_{-n},y_{-n},a)$
is holomorphic in $(x,y,a)$ on $V_+\times D_R$. Therefore,
$a^{\sigma_n+1}h(x_{-n},y_{-n},a)$ vanishes in $a$.

Thus, $\phi_{a,-}^{2^n}=(p(y)-x)^{2^{n-1}}+ah_1(x,y,a)$.
\end{proof}

\noindent We denote $K_0^-$ and $J_0^-$ to be $C(p)$. We set
$U_0^-=\mathbb C^2\backslash J_0^-$. The previous two statements
justify these notations.

\section{Green's functions}\label{sec:Green_function}

The next two lemmas are from \cite{HOV1}. We provide the proof of
Lemma \ref{lem:Green_function} to make the paper self-contained.

\begin{lemma}\label{lem:Green_function} The function

\begin{equation}\label{Green_function}
G_a^+(x,y)=\lim_{n\to \infty}\log^+|f_a^n(x,y)|
\end{equation}

\noindent is well-defined in $\mathbb C^2\times D_R$. It satisfies
the functional equation

\begin{equation}\label{functional_equation}
G_a^+(f_a)=2 G_a^+.
\end{equation}

\noindent Moreover, it is pluriharmonic in $U_a^+$.
\end{lemma}

\begin{proof}

We define, the function $G_a^+=\log |\phi_{a,+}|$ on $V_+$. We use
the functional equation (\ref{functional_equation}) to extend it
to $U_a^+$. We also set $G_a^+(x,y)=0$ for $(x,y)\in J_a^+$.
$G_a^+$, defined this way, is pluriharmonic in $V_+$, since it is
a logarithm of a holomorphic non-zero function $\phi_{a,+}$. It is
pluriharmonic in $U_a^+$, since for each $n$, on $f_a^{-n}(V_+)$,
$G_a^+$ is a pull-back of a pluriharmonic function by a
holomorphic change of coordinates.

Moreover, notice that the function defined this way satisfies
(\ref{Green_function}) and there is a unique function that
satisfies (\ref{Green_function}).
\end{proof}

\begin{lemma}
The function

\begin{equation}
G_a^-(x,y)=\lim_{n\to \infty}|f_a^{-n}(x,y)|+\log|a|
\end{equation}

is well-defined in $\mathbb C^2\times D_R^*$. It satisfies the
functional equation

\begin{equation}\label{functional_equation2}
G_a^-(x,y)\circ f_a^{-}=2 G_a^--\log|a|.
\end{equation}

\noindent Moreover, it is pluriharmonic on $U_a^-$.
\end{lemma}

Equation (\ref{functional_equation2}) is sometimes more
conveniently written

$$(G_a^{-}-\log|a|)\circ f_a^{-1}=2(G_a^{-}-\log|a|).$$

We set

\begin{equation}
G_0^-(x,y)=\log|\phi_{0,-}(x,y)|=\left\{\begin{array}{ll}\frac{1}{2}\log|p(y)-x|\
& \mbox{for}\ (x,y)\not \in C(p);\\

-\infty & \mbox{for}\ (x,y)\in C(p).\end{array}\right.
\end{equation}

Hubbard \& Oberste-Vorth proved that the Green's functions are
continuous when $f_a$ is non-degenerate and the same argument
gives continuity in $x,y$ and $a$ for $G_a^+$ when $a=0$.
Lyubich-Robertson \cite{LR} extend this to $G_a^-$ when $a=0$.

\begin{lemma}[\cite{LR}] The functions $G_a^+(x,y)$ and
$G_a^-(x,y)$ are continuous in $x,y$ and $a$ for $a\in D_R$.
\end{lemma}

\begin{proof} It follows by the same argument as is used in
\cite{HOV1} except in the case of $G_a^-$ when $a=0$. For
$(x',y')\not \in C(p)$ the continuity of $G_a^-$ at $(x',y')$ and
$a=0$ follows from Lemma \ref{lim_value2}. For $(x',y')\in C(p)$
more work is required. If we restrict $G_a^-$ to the slice $a=0$,
then we have shown continuity, so we will assume for most of the
rest of this proof that $a\neq 0$ (so $f_a^{-1}$ is defined).

Let us fix $M>0$, and find a neighborhood $U$ of $(x',y')$ so that
for all $(x,y)\in U$, $G_a^-(x,y)<-M$

If $(x,y)\in J_a^-$, then it is enough to require that
$|a|<e^{-M}$.

If $(x,y)\in U_a^-$, then $f_a^{-n}(x,y)\in V_-$ for all $n>n_0$.

If $f_a^{-n}(x,y)\in V_-$, then by Lemma \ref{existence_phi-}
$B^{-1}<\left|\frac{\phi_{a,-}(x_{-n}, y_{-n})}{y_{-n}}\right|<B$.

Therefore, $G_a^-(x,y)<\frac{1}{2}\log
B+\frac{1}{2^n}\log|a^{\sigma_n} y_{-n}|$.

\begin{equation}\label{rec_rel} y_{-n}=\frac{1}{a}\left(p(y_{-(n-1)}\right)-y_{-(n-2)})
\ \mbox{for}\ n\geq 2.\end{equation}

$$y_0=y,\ y_{-1}=\frac{p(y)-x}{a}$$

We wish to estimate $a^{\sigma_n} y_{-n}$ in terms of $y_0$ and
$y_{-n}$.

It is convenient to introduce a new variable
$z_{-n}=a^{\sigma_n}y_{-n}$ and a notation $p~(~x~,~y~)~=~y^2~
p~(~\frac{x}{y}~)~.$

In these new notations the recurrence relation (\ref{rec_rel})
takes form

\begin{equation}
z_{-n}=p(z_{-(n-1)},
a^{\sigma_{n-1}})-a^{\sigma_n-\sigma_{n-2}-1}z_{n-2},
\end{equation}

\noindent where $z_0=y$, $z_{-1}=p(y)-x$.

$|z_{-n}|\leq 2\max(|p(z_{-(n-1)}|, |a|^{\sigma_{n-1}}),
|a|^{2^{n}+2^{n-1}-1}|z_{n-2}|)$

Using the estimate $p(x,y)\leq C\max(x^2, a^2)$, where $C$ is a
constant, we get

$|z_{-n}|\leq 2\max(|z^2_{-(n-1)}|, |a|^{2\sigma_{n-1}},
|a|^{2^n+2^{n-1}-1}|z_{n-2}|)$

Consider a neighborhood

\begin{equation}
U(\epsilon)=\{(x,y,a)|\ |p(y)-x|<\epsilon^2,
|a|<\min(2C\epsilon^2, \frac{1}{2}), |a||y|<C\epsilon^2\}
\end{equation}

By induction one can show that if $(x,y,a)\in U(\epsilon)$, then
$z_{-n}<(2C)^{\sigma_n}\epsilon^{2^n}$.

Therefore, for $(x,y,a)\in U(\epsilon)$

\begin{equation}
G_a^-(x,y,a)<\log(|2C|\epsilon)
\end{equation}

Therefore, for small enough $\epsilon$, $G_a^{-}(x,y,a)<-M$.

Intersecting $U(\epsilon)$ with the $e^{-M}$-neighborhood in
$a$-variable, we get the desired neighborhood.

\end{proof}

\section{Description of the foliations and the critical locus in
terms of $\phi_{a,+}$ and $\phi_{a,-}$}\label{sec:critical_locus}

Note that

$$G_a^+(x,y)=\log|\phi_{a,+}(x,y)|;$$

\noindent Therefore, the foliation $\phi_{a,+}=\mbox{const}$ on
$V_{+}$ is exactly ${\cal F}_a^+$. It can be propagated by
dynamics to the rest of $U_a^+$.

The same way,

$$G_a^-(x,y)=\log|\phi_{a,-}(x,y)|.$$

The foliation $\phi_{a,-}=\mbox{const}$ on $V_-$ is ${\cal F}_a^-$
and it propagates by dynamics to $U_a^-$

$${\cal U}^{\pm}=\{(x,y,a) |\ a\in D_R, (x,y)\in U_a^\pm\}
$$

Below we show that the critical locus is defined by a global
holomorphic form on ${\cal U}^+\cap {\cal U}^{-}$. Therefore, it
is a proper analytic subset of ${\cal U}^{+}\cap {\cal U}^{-}$.

The forms $d \log \phi_{a,+}$ and $d \log \phi_{a,+}$ are
well-defined and holomorphic in ${\cal U}^+$ and ${\cal U}^-$
correspondingly.

Critical locus $\cal C$ is given by the zeroes of the form

$$w(x,y,a)dx\wedge dy \wedge da = d\log \phi_{a,+}\wedge d \log \phi_{a,-}\wedge
da$$

\noindent that is holomorphic in ${\cal U}^{+}\cap {\cal U}^{-}$.
When we prefer to think of $w(x,y,a)$ as a function of two
variables we use the notation $w_a(x,y)=w(x,y,a)$.

Let ${\cal C}_{a_0}={\cal C}\cap \{a=a_0\}$.

We recall a definition of a proper analytic subset:

\begin{definition} $A$ is a proper analytic subset of a complex
manifold $M$, if for every point $x\in M$, there exist a
neighborhood $U$ and a set of functions $f_1,\dots, f_n$ so that
$$f_1=\dots=f_n=0$$
define $A$ in $U$.
\end{definition}

\begin{lemma}\label{analiticity} ${\cal C}$ is a proper analytic
subset of ${\cal U}^{+}\cap {\cal U}^{-}$.
\end{lemma}

\begin{proof} ${\cal C}$ is defined by the zeroes of the form that is
holomorphic in ${\cal U}^{+}~\cap~{\cal U}^{-}$.
\end{proof}

\subsection{Degenerate critical locus.}

As $a\to 0$ the H\'{e}non mapping degenerates to $(x,y)\mapsto
(p(x),x)$. The foliation ${\cal F}_a^+$ is defined by the form
$d\phi_{a,+}$ on $U^+_a$. Note that $d\phi_{a,+}\neq 0$ on $U_a^+$
for $a\in D_R^*$. Therefore, the foliation ${\cal F}_a^+$ for
$a\in D_R^*$ is nondegenerate. The liming foliation ${\cal F}_0^+$
is degenerate. It has double leaves.

Below we explain what it means for the foliation  to have a double
leaf.

Recall that $U_a^+=\sum_n f_a^{-n}(V_+)$, on each $f_a^{-n}(V_+)$
the foliation ${\cal F}_{a}^+$ is defined by the level sets of the
function $\phi_{a,+}^{2^n}$.

Suppose that a foliation $\cal F$ in $\Omega\subset\mathbb C^2$ is
defined by the level sets of a function $\phi$. Suppose that in a
neighborhood of each point $(x,y)$ one can choose local
coordinates $(u,t)$, so that $\phi=u^n$.

\begin{definition} We say that a leaf $L$ of the foliation $\cal F$
is double if in a neighborhood of a point $(x,y)\in L$,
$\phi=u^2$.
\end{definition}

\begin{note} The definition depends on the defining function $\phi$ and does not depend on the choice of the
local parameter $u$, nor on a point $(x,y)\in L$.
\end{note}

\begin{lemma} The foliation ${\cal F}_0^+$ in $U_0^+$ is a vertical foliation
$x=const$ with leaves $x=p^{-k}(0)$ being double for all $k\geq
0$.
\end{lemma}

\begin{proof} For each point $x_0$ there exists $n$, such that in a neighborhood of the line $x=x_0$
the foliation ${\cal F}_0^+$ is determined by the level sets of
the function $\phi^{2^n}_{0,+}$.

Since $\phi^{2^n}_{0,+}(x,y)=b^{2^n}_p(x)$, the foliation is
vertical.

The multiple leaves appear when $\left(b^{2^n}_p\right)'(x)=0$.
$\left(b^{2^n}_p\right)'(x)=0$ if and only if $p^n~(~x~)~=~0~.$
Moreover, all zeros of $\left(b_p^{2^n}\right)'$ are
non-degenerate. Therefore, all the leaves $x=p^{-n}(0)$ are
double.
\end{proof}

The foliation ${\cal F}_a^-$ in $U^-a$ for $a\in D_R^*$ is defined
by the form $d\phi_{a,-}^{2^n}$. By Lemma \ref{lem:swell} as $a\to
0$, $U_a^-$ swell to $\mathbb C^2\backslash C(p)$. The form
$d\phi_{a,-}$ extends to $\mathbb C^2\backslash C(p)$ as a
holomorphic form. Let ${\cal F}_0^-$ denote the limiting foliation
defined by the form $d\phi_{0,-}$.

\begin{lemma} The foliation ${\cal F}_0^-$ in $U_0^-$ consists of the leaves
$\{p(y)-x=const\}$.
\end{lemma}

\begin{proof} By Lemma \ref{limiting_value} $\phi_{0,-}=p(y)-x$ and is defined in $U_0^-$. The statement of
the Lemma immediately follows.
\end{proof}

\begin{corollary}\label{degenerate-critical-locus} ${\cal C}_0=\left[\{y=0\}\bigcup_k\{x=p^{-k}(0)\}\right]\cap U_0^+\cap U_0^-$
\end{corollary}

\begin{figure}[h]
\centering
\psfrag{J+}{$J_0^+$}\psfrag{J-}{$J_0^-$}\psfrag{x}{$x$}\psfrag{y}{$y$}\psfrag{x=p1}{$x=p^{-1}(0)$}\psfrag{x=pk}{$x=p^{-k}(0)$}
\psfrag{x=pm}{$x=p^{-m}(0)$}
\includegraphics[height=5cm]{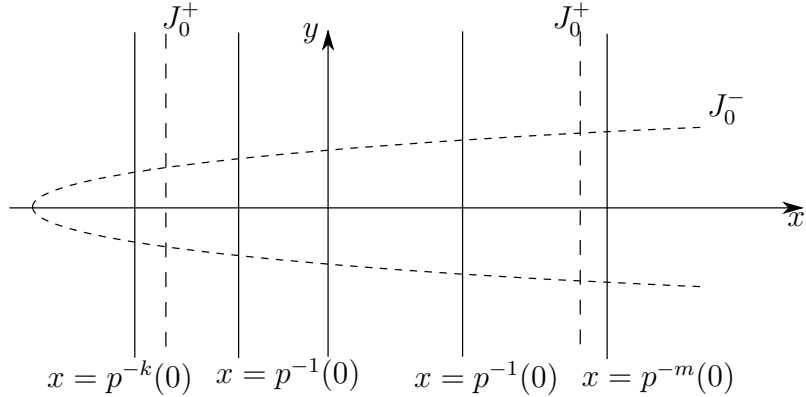}
\caption{The degenerate Critical Locus}
\end{figure}

\section{Critical locus near infinity.}\label{infinity}

The goal of this section is to calculate the critical locus at a
neighborhood of $x=\infty$ in the compactification of $\mathbb
C^2$ given by $\mathbb C\mathbb P^1\times \mathbb C\mathbb P^1$.

First, we extend the foliation ${\cal F}_a^+$ to the line
$x=\infty$ in $\mathbb C\mathbb P^1\times \mathbb C\mathbb P^1$
compactification of $\mathbb C^2$. Let

$$\hat{V}_{+}= \{(x,y)\in \mathbb C\mathbb P^1\times \mathbb
C\mathbb P^1|\ (x,y)\in V_+ \ \mbox{or}\ \{x=\infty, y\neq
\infty\}\}$$




Let $t=\frac{1}{x}$, $v=\frac{1}{y}$.

\begin{lemma}The function $\frac{\phi_{a,+}}{x}$ extends as a holomorphic function to ${\hat V}_+\times D_R$

Moreover, $t \phi_{a,+}=1+th_1(t,y,a)$, $\frac{\partial
h_1}{\partial y}=u \tilde{h}_1(t,y,a)$.
\end{lemma}

\begin{proof} By the Riemann Extension Theorem, $\frac{\phi_{a,+}}{x}$ can
be extended to $x=\infty$.

Note that the functions $s_n^+$ and $\frac{\partial
s_n^+}{\partial y}$ vanish in $t$ for all $k$. Thus, in the sum
(\ref{sum1}) every term vanishes in $u$. Therefore, the infinite
sum (\ref{sum1}) vanishes in $u$ as well. And
$\frac{\phi_{a,+}}{x}=1+th_1(t,y,a)$, where $h_1$ is a holomorphic
function in $t,y,a$. The same way one proves, $\frac{\partial
h_1}{\partial y}=t\tilde{h}_1(t,y,a)$.
\end{proof}


$$
\hat{V}_-=\{(x,y)\in \mathbb C\mathbb P^1\times \mathbb C\mathbb
P|\ (x,y)\in V_{-}\ \mbox{or}\ \{y=\infty,x\neq \infty\}\}
$$

\begin{lemma}[\cite{LR}]\label{phi-infinity} The function $\frac{\phi_{a,-}}{y}$ extends
holomorphically to $\hat{V}_-\times D_R$. Moreover,
$v\phi_{a,-}=1+vh_2(x,v,a)$, $\frac{\partial h_2}{\partial
x}=v\tilde{h}_2(x,v,a)$.
\end{lemma}

\begin{proof} The proof is the same as in the previous lemma.
\end{proof}

In $V_+\cap {\cal D}_{1,-}$ the critical locus is given by the
zeroes of the form

$$\omega = d\phi_{a,+}\wedge d\phi^2_{a,-}\wedge da$$

For the next lemma we fix $(t,y)$-coordinates in $\hat{V}_+$.

\begin{lemma}\label{lemma_infinity} 
The critical locus extends holomorphically to $\left(\hat{V}_+\times D_R\right)\cap
{\cal D}_{1,-}$. The point $(0,y,a)\in {\cal C}_a\cap
\hat{V}_+\cap {\cal D}_{1,-}$ iff $y=0$. The tangent line to the
critical locus at $(0,0,a)$ is given by $2dy+Cdt=0$ with $C$
depending on $a$.
\end{lemma}

\begin{proof}
$$d\phi_{a,+}=-\frac{dt}{t^2}+\frac{\partial h_1}{\partial
t}dt+\frac{\partial h_1}{\partial y}dy+\frac{\partial
h_1}{\partial a}da=\frac{1}{t^2}\left(-1+t^2\frac{\partial
h_1}{\partial t}dt+t^3\tilde{h_1}dy+t^2\frac{\partial
h_1}{\partial a}da\right)$$

$$\phi^2_{a,-}(x,y)=a
\phi_{a,-}\circ f^{-1}_a(x,y)=a\phi_{a,-}\left(y,
\frac{p(y)-x}{a}\right)=\frac{tp(y)-1}{t}+ah_2(y,\frac{at}{tp(y)-1},a)$$

$$d\phi_{a,-}=\left(p'(y)+a\frac{\partial h_2}{\partial y}(y,
\frac{at}{tp(y)-1},a)+\frac{a^2t^2p'(y)}{(tp(y)-1)^2}\frac{\partial
h_2}{\partial v}(y,\frac{at}{tp(y)-1})\right)dy$$

$$+\left(\frac{1}{t^2}+\left[\frac{a}{tp(y)-1}+\frac{atp(y)}{(tp(y)-1)^2}\right]\frac{\partial
h_2}{\partial v}\right)dt+h_3(t,y,a)da,$$

\noindent where $h_3(t,y,a)$ is some holomorphic function in
$t,y,a$.

By Lemma \ref{phi-infinity}, $\frac{\partial h_2}{\partial
x}=v\tilde{h}_2$. Hence $\frac{\partial h_2}{\partial
x}(y,\frac{at}{tp(y)-1},a)=\frac{at}{tp(y)-1}\frac{\partial
h_2}{\partial x}(y,\frac{at}{tp(y)-1})$. Thus, there are
holomorphic functions $h_4,$ $h_5$ on ${\cal D}_1^-$ so that

$d\phi_{a,-}=(p'(y)+th_4(t,y,a))dy+(\frac{1}{t^2}h_5(t,y,a))dt+h_3(t,y,a)da$

Therefore, there exists a holomorphic function $h_6$ on
$\left(\hat{V}_+\times D_R\right)\cap {\cal D}_{1,-}$ so that

$$\omega t^2=\left(p'(y)+t h_6\right)dy\wedge dt\wedge da$$

\noindent Conclusion follows.
\end{proof}

\begin{corollary} Fix $\epsilon$. There exists $\delta$
so that for $|a|<\delta$ the critical locus in
$$\{|y|\leq \epsilon\}\cap \{|x|\geq \alpha\}$$
\noindent is the graph of a function $y(x)$.
\end{corollary}

\section{Horizontal and vertical invariant cones.}

\subsection{Horizontal cones.}

Fix a domain $B$.

\begin{definition} A family of cones $C_x$ in the tangent bundle to $B$ is
$f_a$-invariant iff for any point $x\in B$, such that $f_a(x)\in
B$, we have $df_a(C_x)\cup C_{f(x)}$.
\end{definition}

\begin{lemma}\label{horizontal} Fix $r'>r$ and $\beta$. For every
$C<\min \{x : G_p(x)\leq \frac{r'}{2}\}$ there exists $\delta$
such that for all $|a|<\delta$ the family of horizontal cones
$|\xi|>C|\eta|$ is $f_a$-invariant in $\{G_a^+(x,y)\leq r'\}\cap
\{|y|\leq \beta\}$ (where $(\xi,\eta)\in T_{(x,y)}\mathbb C^2$).
\end{lemma}

\begin{note} Note that we chose the box $\{G_a^+(x,y)\leq r'\}\cap
\{|y|\leq \beta\}$ so that the tip of the parabola does not belong
to the box. This allows us to have an invariant horizontal family
of cones.
\end{note}

\begin{proof} $Df_a(x,y)=\left(\begin{array}{ll} 2x & a\\ 1 & 0
\end{array}\right)$

Let $\left(\begin{array}{l}\xi_1\\\eta_1
\end{array}\right)=\left(\begin{array}{cc} 2x & a\\ 1 & 0
\end{array}\right)\left(\begin{array}{l}\xi \\ \eta\end{array}\right)$

We want to find $C$ such that $$|\xi|>C|\eta| \Rightarrow |2x
\xi+a \eta|>C|\xi|.$$

Take $C=\min(2x)-\epsilon,$ where $b_p(x)\leq \frac{r'}{2}$ and
$\epsilon$ is any number. Then $\delta_2=\epsilon \max(2x)$.
\end{proof}

\subsection{Vertical cones.}

\begin{lemma}\label{vertical} Fix $r'\geq r,$ $\alpha,$ $C$. There
exists $\delta$ such that for all $|a|<\delta$ the family of cones
$|\xi|<C|a||\eta|$ is $f_a^{-1}$-invariant in $\{G_a^+(x,y)\leq
r'\}\cap\{|y|\leq \alpha\}$, $(\xi,\eta)\in T_{(x,y)}\mathbb C^2$.
\end{lemma}

\begin{proof} $D f^{-1}_{a}=\left(\begin{array}{cc}0 & 1 \\ -\frac{1}{a} & \frac{2y}{a}\end{array}\right)$

Let $\left(\begin{array}{c}\xi_1 \\ \eta_1\end{array}\right)=\left(\begin{array}{cc}0 & 1 \\
-\frac{1}{a} &
\frac{2y}{a}\end{array}\right)\left(\begin{array}{c}\xi \\
\eta\end{array}\right)$

$\left|\frac{\xi_1}{\eta_1}\right|=\frac{|a|}{|-\frac{\xi}{\eta}+2y|}$

Suppose $(x,y)= f^{-1}_a(u,v)$, where $\{(u,v) |\ G_a^+(u,v)\leq
r', |v|\leq |\alpha|\}$, then $|y|>C_1$.

Therefore, when $|a|<\delta$, $|2y-\frac{\xi}{\eta}|\geq |C_1| $
\end{proof}

\section{Description of ${\cal F}_a^-$.}\label{F-}

In this section we give a description of ${\cal F}_a^-$ in

$$
W \cap \{|p(y)-x|\geq |a|\alpha\}.
$$

We choose $\alpha$ in the definition of $W$ such that the
description of ${\cal F}_a^-$ is especially nice.

The function $\phi^2_{a,-}$ is well-defined in this region, so it
is natural to expect that the foliation
$${\cal F}_a^-=\{\phi^2_{a,-}=\mbox{const}\}$$

\noindent is close to the foliation

$${\cal F}_0^-=\{\phi^2_{0,-}=p(y)-x=\mbox{const}\}.$$

\noindent The only region where it really needs to be checked is
when we approach $a|\alpha|$-neighborhood of $C_p$.

We also prove that the leaves of ${\cal F}_a^-$ that intersect the
boundaries

$$\{G_a^+\leq r\}\cap \{|p(y)-x|=a|\alpha|\}$$

$$\{G_a^+\leq r\}\cap \{|y|\leq |\alpha|\}$$

\noindent are horizontal-like. In order to guarantee this we will
need to choose appropriate $\alpha$.

We start by fixing preliminary $\tilde{\alpha}$ such that
conditions (\ref{1}) and (\ref{2}) are satisfied.

Notice that $\phi^2_{a,-}$ is well-defined in
$$f_a(V_-)=\{|p(y)-x|\geq |a|\tilde{\alpha}, |p(y)-x|\geq
|a||y|\}$$.

Therefore, the domain of definition of $\phi^2_{a,-}$ in $W$ is

$$f_a(V_{-})\cap W=\{|p(y)-x|\geq |a|\tilde{\alpha} \}.$$

\begin{lemma} There exists $\alpha$ such for all
$a\in D_R$

$$\min \{|\phi^2_{a,-}(x,y)| :\ (x,y)\in W, |p(y)-x|=
\alpha|a|\}>$$

$$\max \{|\phi^2_{a,-}(x,y)| :\ (x,y) \in W,
|p(y)-x|=\tilde{\alpha}|a|\}.$$
\end{lemma}

\begin{proof}

$\phi^{2}_{a,-}(x,y)=a\phi_{a,-}(y,\frac{p(y)-x}{a})$.

\noindent $\phi_{a,-}(x,y)\sim y$ as $y\to \infty$, Therefore,
$$\min \{|\phi_{a,-}(x,y)| :\ |x|<\theta,|y|=\alpha \}
>\max \{|\phi_{a,-}(x,y)| :\ |x|<\theta, |y|=\tilde{\alpha}\}
$$

\noindent for big enough $\alpha$. Take
$\theta=\max(p^{-1}(D_{|x|+|a|\alpha}))$.

\end{proof}

\begin{corollary} For a point $q\in W \cap
\{|p(y)-x|\geq |a|\alpha\}$ a connected component of a leaf $L_q$
of the foliation ${\cal F}_a^-$, passing through a point $q$,
stays outside of $\tilde{\alpha}|a|$-neighborhood of $C_p$.
\end{corollary}

\begin{lemma} For all $a\in D_R$,
$\frac{\partial \phi_{a,-}^2/\partial y}{\partial
\phi_{a,-}^2/\partial x}$ is $a$-close to $\frac{\partial
\phi_{0,-}^2/\partial y}{\partial \phi_{0,-}^2/\partial x}=-p'(y)$
in $\Omega$.
\end{lemma}

\begin{note} Note that $\frac{\partial \phi^2_{a,-}}{\partial y}, \frac{\partial \phi^2_{a,-}}{\partial
y}$ do not stay bounded in $\Omega$ as $a\to 0$, but their ratio
does.
\end{note}

\begin{proof} We chose $\alpha$ so that the function $\phi^2_{a,-}$ is
well-defined in $\{G_a^+\leq r \}\cap \{|a|\tilde{\alpha} \leq
|p(y)-x|\leq \kappa\}$ and that the leaves of ${\cal F}_a^-$ we
consider do not leave this neighborhood.

Let us show that in this region $\frac{\partial
\phi_{a,-}^2/\partial y}{\partial \phi_{a,-}^2/\partial x}$ is
$a$-close to $\frac{\partial \phi_{0,-}^2/\partial y}{\partial
\phi_{0,-}^2/\partial x}$.

$$\frac{\partial \phi_{a,-}^2}{\partial y}(x,y)=
\frac{\partial}{\partial y}\left(a\phi_{a,-}(y,
\frac{p(y)-x}{a})\right)= a\frac{\partial \phi_{a,-}}{\partial
x}\left(y,\frac{p(y)-x}{a}\right)+$$

$$p'(y)\frac{\partial \phi_{a,-}}{\partial
y}\left(y,\frac{p(y)-x}{a}\right)$$

$$\frac{\partial \phi_{a,-}^2}{\partial
x}(x,y)=\frac{\partial}{\partial x}\left(a
\phi_{a,-}(y,\frac{p(y)-x}{a})\right)=-\frac{\partial
\phi_{a,-}}{\partial y}$$

Let us do a change of coordinates
$(u,y,v)=(p(y)-x,y,\frac{a}{p(y)-x})$. First, we introduce the
$u$-coordinate that measures the distance to parabola. Then we do
a blow-up in each line $y=const$ that blows-up a cone in
$u,a$-coordinates, which corresponds to the compliment of
$|a|\alpha$ neighborhood, to a polydisk in $u,v$ coordinates.

\begin{figure}[h]
\centering
\psfrag{|u|=alpha|a|}{$|u|=\alpha|a|$}\psfrag{1/alpha}{$\frac{1}{\alpha}$}
\psfrag{u}{$u$}\psfrag{a}{$a$}\psfrag{v}{$v$}
\includegraphics[height=5cm]{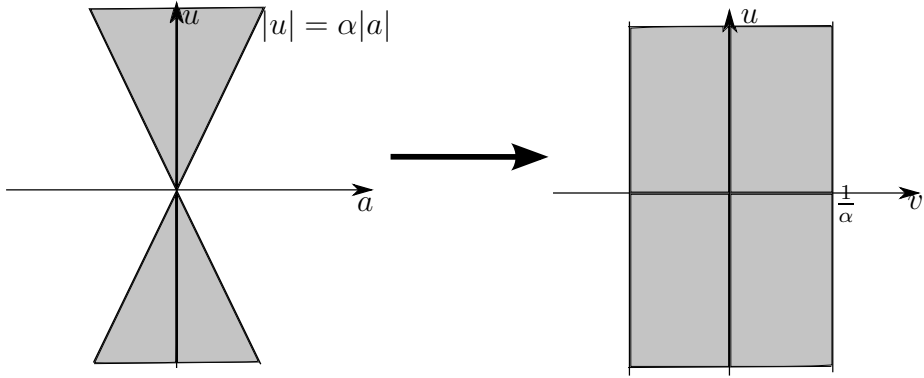}
\caption{The blow-up}
\end{figure}

Denote by $\tilde{\phi}^2_-(u,y,v)=\phi^2_{uv,-}(p(y)-u, y)$. We
show that a function $\frac{\partial \tilde{\phi}^2_{-}/\partial
y}{\partial \tilde{\phi}^2_{-}/\partial x}$ lifts to a holomorphic
function on the blown-up space $\{|y|<\alpha, |u|<\epsilon,
|v|<\frac{1}{\alpha}\}$.

$$a\frac{\partial \phi_{a,-}}{\partial x}(y,\frac{p(y)-x}{a})=uv\frac{\partial \phi_{uv,-}}{\partial x}(y,\frac{1}{v})$$
$$\lim_{u\to 0} uv\frac{\partial \phi_{uv,-}}{\partial
x}(y,\frac{1}{v})=0$$
$$
\lim_{u\to 0} \frac{\partial \tilde{\phi}^2_{-}/\partial
y}{\partial \tilde{\phi}^2_{-}/\partial x}=-p'(y).
$$
Thus, $\frac{\partial \phi_{a,-}^2/\partial y}{\partial
\phi_{a,-}^2/\partial x}$ is $a$-close to $\frac{\partial
\phi_{0,-}^2/\partial y}{\partial \phi_{0,-}^2/\partial x}$ in
$\Omega\cap \{|p(y)-x|\geq |a|\alpha\}$.
\end{proof}

Let $L_q$ be a leaf of foliation ${\cal F}_a^-$ that passes
through a point $q$.

\begin{corollary}\label{horizontal_cones_F-} There exist $\kappa$ and $\delta$ such that for all $|a|<\delta$
and all
$$q\in \{G_a^+\leq r \}\cap \{|p(y)-x|\geq |a|\alpha\}\cap
\{|y|\geq \kappa\}$$ a connected component of $L_q$ is
horizontal-like.
\end{corollary}

\begin{proof} One takes $\kappa$ such that every leaf of ${\cal
F}_0^-$ that intersects $|y|=\kappa$ is horizontal like.
\end{proof}

\begin{corollary}\label{lem1} There exists $\delta$ so that for all
$|a|<\delta$ and all $(x,y)\in \Omega$ the tangent plane to the
foliation ${\cal F}_a^-$ is not horizontal.
\end{corollary}

\section{Description of ${\cal F}_a^+$.}\label{F+}





\begin{definition} We say that a curve $C$ in a domain $B$ is horizontal-like iff there
exists a family of $f_a$-invariant horizontal cones in $B$, such
that the tangent lines to $C$ belong to this family.
\end{definition}

The function $G_0^+(x,y)=G_p(x)$ is self-similar:
$$G_p(p(x))=2G_p(x).$$
Recall that we chose $r$ so that
$$G_p(0)< r< G_p(p(0)).$$
\begin{figure}[h]
\centering
\includegraphics[height=3cm]{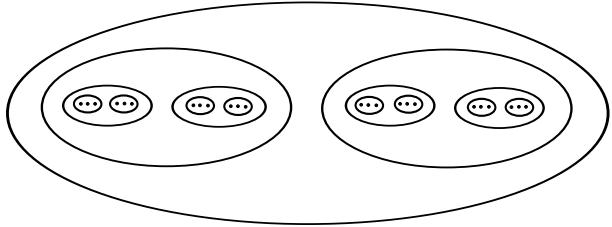}
\caption{Level sets $G_p=\frac{r}{2^n}$.}
\end{figure}

The picture of the level sets of $G_p$ inside
$\{G_p<r\}$ is self-similar. Inside each connected component of
$\{G_p=\frac{r}{2^n}\}$ there are exactly two connected components
$\{G_p=\frac{r}{2^{n+1}}\}$.

There is exactly one critical level in each connected component
$$\frac{r}{2^{n+1}}\leq G_p\leq \frac{r}{2^n}.$$

\begin{figure}[h]
\centering
\includegraphics[height=3cm]{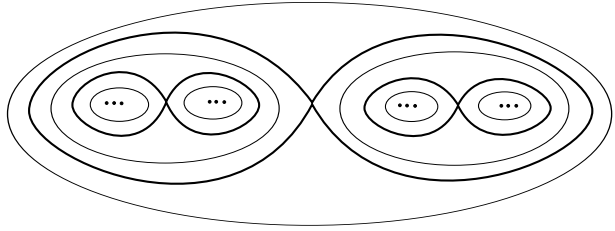}
\caption{Level sets of $G_p$.}
\end{figure}

In this section we show that the picture persists for the level
sets of $G_a^+$ for small enough $a$ on each horizontal-like curve
in a box
$$\{G_a^+\leq r\}\cap \{|y|\leq \beta\}.$$
We choose $\beta=2\max \{p(x) |\quad G_p(x)\leq r\}$ to
apply this construction to the leaves of ${\cal F}_a^-$ that
intersect $\{|y|=\alpha\}$ in $\{G_a^+\leq r\}$.


\begin{lemma} There exists $\delta$ such that for $|a|<\delta$
$\{G_a^+=r\},$ $\{G_a^+=\frac{r}{2}\}$ are non-critical on each
horizontal-like curve inside $|y|<\beta$. Moreover, $\{G_a^+=r\}$
has one connected component, $\{G_a^+=\frac{r}{2}\}$ has two
connected components.
\end{lemma}

\begin{proof} $G_a^+$ is a function that depends analytically on
$a$ on ${\cal U}_{+}$ for all $a\in D$. Therefore, since the level
sets $\{G_0^+=r\}$ and $\{G_0^+=\frac{r}{2}\}$ on $y=b$ are
non-degenerate, they remain non-degenerate for small enough $a$.
The same is true for the number of components.
\end{proof}

\begin{lemma}
There exists $\delta$ so that for all $|a|<\delta$, there is
exactly one critical level of $G_a^+$ between $\frac{r}{2}$ and
$r$ on each horizontal-like curve in $|y|<\beta$. The
corresponding critical point is non-degenerate.
\end{lemma}

\begin{proof} Let $C$ be a horizontal-like curve. Then the domain $\{G_a^+\leq r\}$
inside $C$ is parametrized by a planar domain. Therefore, the
index of $\mbox{grad} G_0^+$ along the boundary of
$\{\frac{r}{2}\leq G_0^+\leq r\}$ is well-defined and is equal to
one. The function $G_{a}^+$ depends holomorphically on $a$. Thus,
the index of $\mbox{grad} G_a^+$ along the boundary of
$\{\frac{r}{2}\leq G_a^+\leq r\}$ is one as well for small $a$.
Therefore, there is only one critical point inside and it is
non-degenerate.
\end{proof}

\begin{lemma} There exists $\delta$ so that for all $|a|<\delta$ $T_r=\{G_a^+=r\}\cap W$ is a solid torus,
$\{G_a^+=\frac{r}{2}\}\cap W$ consists of two solid tori
$T^1_{r/2},T^2_{r/2}$ (the core coordinate can be chosen to be
real-analytic, the disk coordinate holomorphic).
\end{lemma}

\begin{proof} $\{G_a^+=r\}=\{(\phi_{a,+},y), |\phi_{a,+}|=r, |y|\leq \alpha\}$.

For $\{ G_a^+=\frac{r}{2}\}\cap W$ the proof goes the same way.
\end{proof}

Take any horizontal-like curve. We want to prove by induction that
inside each component $\{G_a^+=\frac{r}{2^n}\}$ there are exactly
two components $\{G_a^+=\frac{r}{2^{n+1}}\}$, and they are
non-critical. Therefore, there is exactly one critical level in
between, and the corresponding critical point is non-degenerate.

\begin{lemma} There exists $\delta$ so that for all $|a|<\delta$,
the level set $\{G_a^+= \frac{r}{2^n}\}$ on each horizontal-like
curve in $|y|<\beta$ is non-critical.
\end{lemma}

\begin{proof}

$f_a^n(\{G_a^+\leq \frac{r}{2^n}\})$ is horizontal-like, since it
is an image of a horizontal-like curve and it belongs to the box
with $f_a$-invariant horizontal cones.

$f_a^n(G_a^+=\frac{r}{2^n})\in T_r$ and it projects one-to-one to
$x$-axis. Therefore, it is non-critical.

\end{proof}

\begin{lemma} For $|a|<\delta$ on
each horizontal-like curve for every $n$ there are exactly two
level sets $\{G_a^+= \frac{r}{2^{n+1}}\}$ inside $\{G_a^+=
\frac{r}{2^{n}}\}$ and they are non-critical.
\end{lemma}

\begin{proof} For every $n$, $f_a^n(\{G_a^+\leq \frac{r}{2^n}\})$
is a disk that projects one-to-one to $x$-axis with the boundary
on $T_r$. It intersects $T^1_{r/2}$, $T^2_{r/2}$ by a circle each.
On the intersection $\{G_a^+=\frac{r}{2}\}$. This proves the
lemma.
\end{proof}

\begin{corollary} For $|a|<\delta$ on each horizontal-like curve there is only one
critical level $\{G_a^+=r'\},$ where
$\frac{r}{2^{n+1}}<r'<\frac{r}{2^n}$ for each connected component
$\{G_a^+=\frac{r}{2^n}\}$.
\end{corollary}

The next lemma states that the foliation ${\cal F}_a^+$ is not
only vertical-like (projects one-to-one to $y$-axis), but is
uniformly close to vertical on some thickening of
$\{G_a^+=\frac{r}{2^n}\}$.




\begin{lemma}\label{F+vertical} Fix a small $\lambda$. There exists $\delta$ s.t. $\forall |a|<\delta$
$$
\left|\frac{\partial \phi^{2^n}_{a,+}/\partial y}{\partial
\phi^{2^n}_{a,+}/\partial x}\right|<C|a|
$$
\noindent on $\{\frac{r-\lambda}{2^n}\leq G_a^+\leq
\frac{r+\lambda}{2^n}\}$ with $C$ independent on $n$.
\end{lemma}

\begin{proof} On $\{r-\lambda\leq G_a^+\leq r+\lambda\}$ the
inequality follows from the fact that $\phi_{a,+}$ is a
holomorphic function in $a$.

The leaves of foliation ${\cal F}_a^+$ in
$\{\frac{r-\lambda}{2^{n+1}}\leq G_a^+\leq
\frac{r+\lambda}{2^{n+1}}\}\cap W$ are preimages under $f_a^{-1}$
of the leaves of ${\cal F}_a^+$ in $\{\frac{r-\lambda}{2^{n}}\leq
G_a^+\leq \frac{r+\lambda}{2^{n}}\}\cap W$. Therefore, by
induction we check that they belong to $f_a^{-1}$-invariant
vertical cones.
\end{proof}

\section{Critical Locus in $\Omega$.}\label{Omega}

Recall that

$$\Omega_a = \{G_a^+\leq r\}\cap
\{|y|\leq \alpha\}\cap \{|p(y)-x|\geq |a|\alpha\}.$$

\noindent See Figure \ref{fig:Omega}. We will omit the subscript
$a$, when it is clear from the context which H\'{e}non mapping is
under consideration.

In this section we subdivide $\Omega_a$ into countably many
regions. Note that the level sets $\{G_a^+=\frac{r}{2^n}\}$ on
each horizontal line $y=k$, $|k|<\alpha$ depend continuously with
respect to parameter $a$. Therefore, the partition of

\begin{equation}\label{for:partition_Omega}\{\frac{r}{2^{n+1}}\leq G_a^+\leq \frac{r}{2^n}\}\cap
\Omega_a
\end{equation}

\noindent into connected components depends continuously on $a$.

Therefore, it is enough to enumerate the connected components of
$\Omega_0$. We call $\Omega_{0}^{\xi_n}$ the connected component
of

$$\{\frac{r}{2^{n+1}}\leq G_0^+\leq \frac{r}{2^n}\}\cap \Omega_0,$$

\noindent which contains the critical line $x=\xi_n$. We call
$\Omega_a^{\xi_n}$ the connected component of
(\ref{for:partition_Omega}) that is the continuation of
$\Omega_0^{\xi_n}$.

\begin{figure}[h]
\centering
\psfrag{|y|=alpha}{$|y|=\alpha$}
\psfrag{G=r}{$G_a^+=\frac{r}{2^n}$}\psfrag{G=r/2}{$G_a^+=\frac{r}{2^{n+1}}$}
\psfrag{u=aalpha}{$|u|=a\alpha$}\psfrag{x}{$x$}\psfrag{y}{$y$}
\includegraphics[height=8cm]{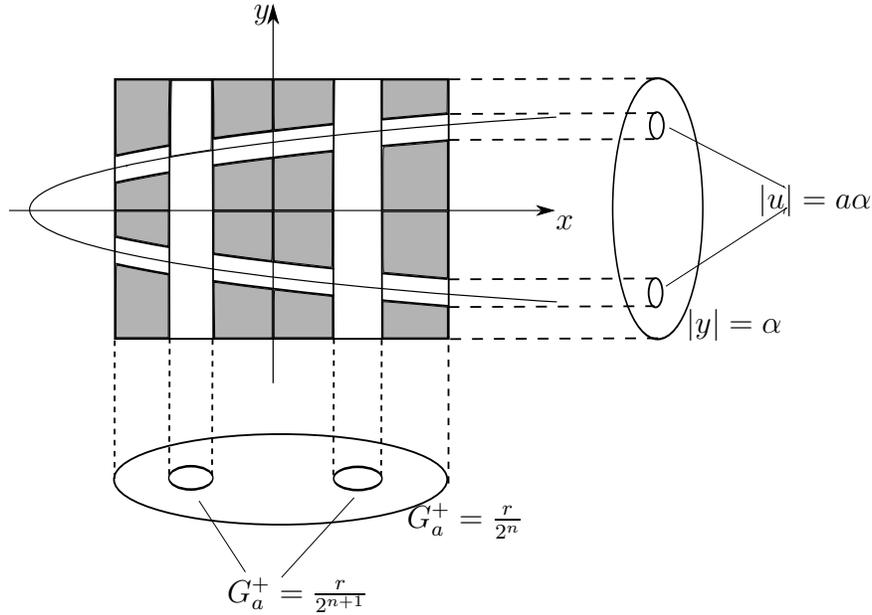}
\caption{Domain $\Omega_a^{\xi_n}$}
\end{figure}

\begin{lemma}\label{x-boundary} Fix a small $\lambda$. There exists $\delta$
(independent on $n$) so that for all $|a|<\delta$ the critical
locus in each connected component of
$$\{\frac{r-\lambda}{2^{n}}\leq G_a^+\leq
\frac{r+\lambda}{2^n}\}\cap \{|y|\leq \alpha\}$$ \noindent is an
annulus which is a graph of function $y(\phi_{a,+})$.
\end{lemma}

\begin{figure}[h]
\centering
\psfrag{G=r}{$G_a^+=\frac{r}{2^n}$}\psfrag{G=r+lambda}{$G_a^+=\frac{r+\lambda}{2^n}$}\psfrag{G=r-\lambda}{$G_a^+=\frac{r-\lambda}{2^n}$}
\includegraphics[height=3cm]{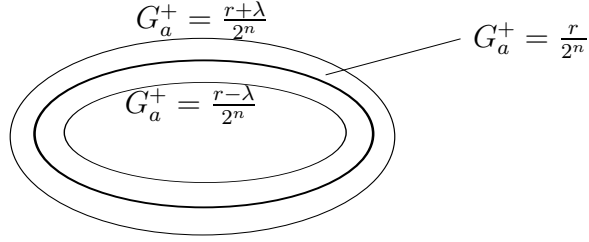}
\caption{The thickening of $G_a^+=\frac{r}{2^n}$}
\end{figure}

\begin{proof} Fix $\gamma$ so that the set $\{
G_0^+\leq r\}\cap \{|y|\leq \gamma\}$ is disjoint from $C_p$. Then
for all $|a|<\delta'$ the set $\{ G_a^+\leq r\}\cap \{|y|\leq
\gamma\}$ is disjoint from $C_p$.

The inverse function theorem implies that $\forall n$ $\exists
\delta_n$ s.t. $|a|<\delta_n$ the critical locus in
$\{\frac{r-\lambda}{2^n}\leq G_a^+\leq\frac{r+\lambda}{2^n}\}\cap
\{|y|\leq \gamma\}$ is an annulus on the component that is a
perturbation of $y=0$.

Since in the region

$$\{\frac{r-\lambda}{2^n}\leq G_a^+\leq \frac{r+\lambda}{2^n}\}\cap \{\gamma<|y|<\alpha\}\cap \{|p(y)-x|\geq a|\alpha|\}$$
the foliation ${\cal F}_a^-$ is $a$-close to ${\cal F}_0^-$ and by
Lemma \ref{F+vertical} ${\cal F}_a^+$ is almost vertical, for
$|a|<\delta''$ there are no points of the critical locus in this
region.

Take some $n$ and fix some connected component
$$\{\frac{r-\lambda}{2^n}\leq G_a^+\leq \frac{r+\lambda}{2^n}\}\cap \{|y|\leq
\gamma\}.$$

Let us show that the critical annuli in this connected component
persists for all $|a|\leq \delta''$. ${\cal C}_a$ is a analytic
set in the region

$$\{(x,y,a) | \ \frac{r-\lambda}{2^n}\leq G_a^+(x,y)\leq \frac{r+\lambda}{2^n}, |y|<\gamma, |a|<\delta''\}$$

There are no zeroes of $\mbox{grad}G_a^-$ on the leaves of ${\cal
F}_a^+$ on curve $|y|=\gamma$. Therefore,  index of $\mbox{grad}
G_a^-$ is constant. At $a=0$ it is equal to $1$. Thus, the
critical annulus in

$$\{(x,y,a) | \ \frac{r-\lambda}{2^n}\leq G_a^+(x,y)\leq \frac{r+\lambda}{2^n}, |y|<\gamma, |a|<\delta''\}$$

\noindent persists.
\end{proof}

\begin{figure}[h]
\centering
\psfrag{|y|=alpha}{$|y|=\alpha$}\psfrag{|y|=kappa}{$|y|=\kappa$}\psfrag{|u|=aalpha}{$|u|=|a|\alpha$}
\includegraphics[height=4cm]{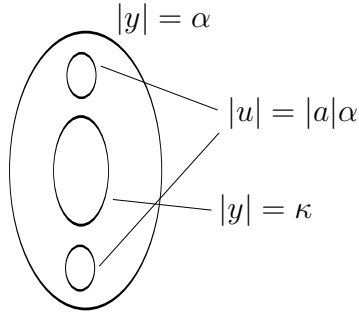}
\caption{The critical locus in $\{|a|\alpha\leq|p(y)-x|\}\cap
\{\frac{r}{2^{n+1}}\leq G_a^+\leq \frac{r}{2^n}\}\cap \{|y|\geq
\kappa\}$}
\end{figure}

\begin{lemma}\label{a-neighborhood} There exist $\kappa$ and $\delta$ such that for all
$|a|<\delta$ the critical locus in each connected component
$\{|a|\alpha\leq|p(y)-x|\}\cap \{\frac{r}{2^{n+1}}\leq G_a^+\leq
\frac{r}{2^n}\}\cap \{|y|\geq \kappa\}$ is an annulus with two
holes and is a graph of  function $y(\phi_{a,-})$.
\end{lemma}

\begin{proof} It follows from Corollary \ref{horizontal_cones_F-}
\end{proof}

We review the notion of the Milnor number of a singularity that we
will need for the proof of the next lemma.

Let $z_0$ be an isolated singular point of a holomorphic function
$f:\mathbb C^n\to \mathbb C$. Let $B_{\rho}$ be a disk of radius
$\rho$ around $z_0$. One can take $\rho$ and $\epsilon$ small
enough, so that the non-singular level sets $\{f=\epsilon\}\cap
B_{\rho}$ have homotopy type of a bouquet of spheres of dimension
$(n-1)$ (\cite{AGVII}, Section 2.1).

\begin{definition} The number of spheres in $\{f=\epsilon\}\cap
B_{\rho}$ is called the {\it Milnor number} of the singular point
$z_0$.
\end{definition}

\begin{lemma}[\cite{AGVII}, Section 2.1]\label{lem:Milnor_number} The Milnor number of the singular point $z_0$
is equal to ${\cal O}_{z_0}/<f_1,\dots,f_n>$, where ${\cal
O}_{z_0}$ are functions regular in a neighborhood of $z_0$ and
$<f_1,\dots, f_n>$ is the ideal in ${\cal O}_{z_0}$, generated by
the functions $f_1,\dots,f_n$.
\end{lemma}

\begin{lemma}\label{description-omega} For $|a|<\delta$ the critical locus
in $\Omega_a$ is a smooth curve. In $\Omega_a^{\xi_n}$ it is a
connected sum of two disks $D_1$ and $D_2$ with two holes. The
boundary of $D_1$ belongs to $\{|y|=\alpha\}$, and the holes of
$D_1$ have boundaries on $\{|u|=|a|\alpha\}$. The boundary of
$D_2$ belongs to $\{G_a^+=\frac{r}{2^n}\}$ and the holes to
$\{G_a^+=\frac{r}{2^{n+1}}\}$.
\end{lemma}

\begin{figure}[h]
\centering
\psfrag{deg}{\it The degenerate critical locus}
\psfrag{nondeg}{\it The nondegenerate critical locus}
\psfrag{y=alpha}{$|y|=\alpha$}\psfrag{Gp=r}{$G_p=\frac{r}{2^n}$}\psfrag{Gp=r/2}{$G_p=\frac{r}{2^{n+1}}$}
\psfrag{Ga=r}{$G_a^+=\frac{r}{2^n}$}\psfrag{Ga=r/2}{$G_a^+=\frac{r}{2^{n+1}}$}\psfrag{u=aalpha}{$|u|=|a|\alpha$}
\includegraphics[height=7cm]{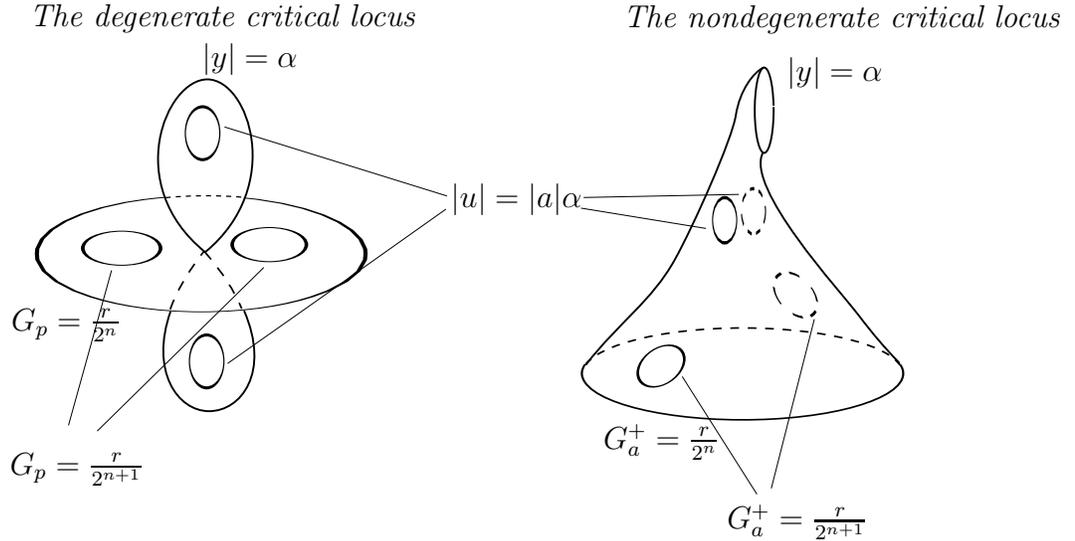}
\caption{The critical locus in $\Omega_a^{\xi_n}$}
\end{figure}

\begin{note} As $a\to 0$, the curve
degenerates to $(x-\xi_n)y=0$. The holes of $D_1$ degenerate to
points $(0,\xi_{n+1}),$ $(0,\xi'_{n+1}),$ where
$p(\xi'_{n+1})=p(\xi'_{n+1})=\xi_n$. The holes of $D_2$ tend to
circles $\{(x, 0)|\ G_p(x)=\frac{r}{2^{n+1}}\}$.
\end{note}

\begin{proof} Fix $\xi_n$. The critical locus in $\Omega_0^{\xi_n}$
is a singular curve. It is a union of two intersecting lines $y=0$
and $x=\xi_n$. By Lemma \ref{lem:Milnor_number} the Milnor number
of this singularity is 1. The critical locus ${\cal C}_a$ is the
zero level set of $w_a(x,y)$. By Lemmas \ref{x-boundary},
\ref{a-neighborhood} there exists $\delta$ so that the critical
locus in $\Omega_{a}^{\xi_n}$ for $|a|<\delta$ is transverse to
the boundary. Thus, there is $\epsilon$ so that the level sets of
$w_a(x,y)=c$, $|c|<\epsilon$ are transverse to the boundary of
$\Omega_a^{\xi_n}$ for all $|a|<\delta$. Let us consider
$(w_a(x,y),a):\Omega_a\times D_{\delta} \to \mathbb C\times
D_{\delta}$. By Ehresmann's Fibration Theorem the level sets of
this function form a locally trivial fibration. Thus, they are
homeomorphic one to the other. By Lemmas \ref{x-boundary} and
\ref{a-neighborhood}, the critical locus intersects the boundary
exactly as prescribed. Therefore, it is enough to show that the
critical locus ${\cal C}_a\Omega_a^{\xi_n}$ for $|a|<\delta$ is
non-degenerate.

Suppose there are critical level sets of the function $w_a(x,y)$.
Then the Milnor number of the corresponding singularity is 1.
Therefore, the singularity is locally an intersection of two
disks. It is easy to see that one of this curves should be a
'horizontal' curve, i.e. project one-to-one to $y$-axis. Its
boundary belongs to $\{|y|=\alpha\}$. This curve necessarily
intersects the 'vertical' curve on which the tangent line to
${\cal F}_a^+$ are horizontal. But then on the point of
intersection the tangent line to the foliation ${\cal F}_a^-$ is
horizontal. That is impossible by Lemma \ref{lem1}.

Therefore, the critical locus in $\Omega_a^{\xi_n}$ is noncritical
for $a\neq 0$. The conclusion of the lemma follows.
\end{proof}

\section{Extension of the critical locus up to $|a|\alpha$-neighborhood of
parabola.}\label{ext-a-neighborhood}

\begin{lemma}\label{lemma-ext-a-neighborhood} There exists $\delta$ so that for all $|a|<\delta$
the component of the critical locus, that is a perturbation of
$y=0$ can be extended up to $|a|\alpha$-neighborhood of parabola
as a graph of function $y(x)$.
\end{lemma}

\begin{proof}

The domain of definition of $\phi^2_{a,-}$ is $f_a(V_-)=\{(x,y)\
|\ |p(y)-x|\geq |a|\alpha, |p(y)-x|\geq |a||y|\}$.

Therefore, in $W$ the domain of definition of $\phi_{a,-}^2$ is
$\{|p(y)-x|\geq |a|\alpha\}$.

Denote $u=p(y)-x$.

Consider new variables $(u,y,v)=(u,y,\frac{a}{u})$. Denote by
$\pi$ the projection
$$\pi:(u,y,v) \mapsto (u,y,uv).$$

Notice that $\pi^{-1}$ blows up a point $u=0$ on each line
$y=y_0$.

Let us prove that one can extend the critical locus to
$$S=\{|u|<\beta, |y|<\epsilon, |v|<\frac{1}{\alpha}\}.$$

Note that $\phi_{a,+}$ can be extended to $S$ since it's a
well-defined holomorphic function in $\pi(S)$.
$\phi_{a,-}^2(x,y)=uv\phi_{uv,-}(y, 1/v)$.

By [LR] $\frac{\phi_{a,-}(x,y)}{y}$ extends to be a holomorphic
function to a neighborhood of $y=\infty$.

Therefore $\phi_{a,-}^2$ extends to $S$. Moreover, notice that on
blown-up lines $\phi_{uv,-}^2=0$.

The critical locus is given by the zeroes of the function
$$w=\frac{d\phi_{a,+}\wedge d\phi_{a,-}^2\wedge da}{dx \wedge dy
\wedge da}.$$

Let
$$\tilde{w}=-\frac{d\phi_{uv,+}\wedge d\phi_{uv,-}^2\wedge dv}{u
du\wedge dy \wedge dv}.$$

Notice that $\tilde{w}=w\circ \pi$.

$$\tilde{w}=-uv \frac{\partial \phi_{uv,+}}{\partial x}(p(y)-u,
y)\frac{\partial \phi_{uv,-}}{\partial x}(y,\frac{1}{v})+$$

$$\left(\frac{\partial \phi_{uv,+}}{\partial
x}(p(y)-u,y) p'(y)+\frac{\partial \phi_{uv,+}}{\partial
y}(p(y)-u,y)\right)\frac{\partial \phi_{uv,-}}{\partial
y}(y,\frac{1}{v})$$

Note that $\frac{\phi_{a,-}}{y}=1+aH(x,\frac{1}{y},a)$, where $H$
is a holomorphic function in some neighborhood of $(x,0,0)$.

$\lim_{u\to 0} v\frac{\partial \phi_{uv,-}}{\partial x}=0$

$\lim_{u\to 0} \frac{\partial \phi_{uv,-}}{\partial y}=1$

$$\tilde{w}(0,y,v)=b'_p(p(y))p'(y)$$

Note that for all $|v|<\frac{1}{\alpha}$ $|y|<\alpha$
$\tilde{w}(0,y,v)=0$ only when $y=0$ and the zero is not multiple.
Therefore, by Weierstrass theorem for $v$, $y=g(u,v)$.
\end{proof}











\section{Description of the critical locus.}\label{final}

Fix some $\epsilon$. Denote

$$\hat{\Omega}_1=\{(x,y)\in \mathbb CP^2 |\ |y|<\epsilon,
|p(y)-x|\geq |a|\alpha, G_a^+(x,y)\geq r\}.$$

\begin{figure}[h]
\centering
\psfrag{omega}{$\Omega$} \psfrag{1Omeg}{$\hat{\Omega}_1$}
\psfrag{x=infty}{$x=\infty$} \psfrag{|a|alpha}{$|a|\alpha$}
\includegraphics[height=5cm]{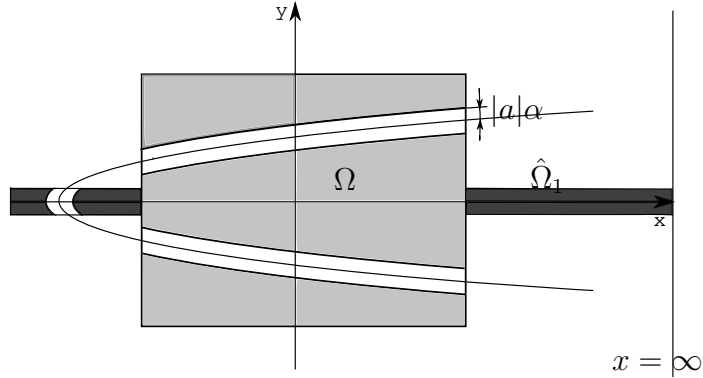}
\caption{Domain $\hat{\Omega}_1$}
\end{figure}

\begin{lemma} There exists $\delta$ such that
$\forall |a|<\delta$ the critical locus ${\cal C}_a$ in
$\hat{\Omega}_1$ is a punctured disk, with a hole removed. The
puncture is at the point $(\infty, 0)$, the boundary of the hole
belongs to $\{|p(y)-x|=|a|\alpha\}$.
\end{lemma}

\begin{proof} By Lemma \ref{degenerate-critical-locus} the
degenerate critical locus ${\cal C}_0$ in $\hat{\Omega}_1$ is
$y=0$, with point $x=c$ removed. By Lemma \ref{lemma_infinity} it
persists in some neighborhood of $x=\infty$. By the inverse
function theorem it can be extended as a graph of function $y(x)$
to $\{~G_a^+~(x,y)~\geq~r~\}$ excluding $\epsilon$-neighborhood of
$C_p$. By Lemma~\ref{lemma-ext-a-neighborhood} it can be extended
up to $a|\alpha|$-neighborhood of $C_p$.
\end{proof}

In the following $3$ lemmas we show that every component of the
critical locus intersects $\Omega$. It follows from Lemma
\ref{description-omega} that it consists of one component.

\begin{lemma}\label{boundary} Let $C_a$ be a component of the critical locus ${\cal C}_{a}$. Then
there exists a point on $\partial C_a$ that belongs to $J_a^+\cup
J_a^-$.
\end{lemma}

\begin{proof} Consider $\left(G_a^+ + G_a^-\right)$. This function is pluriharmonic and strictly positive
in $U_a^+\cap U_a^-$. Therefore, $\inf \left(G_a^+ + G_a^-\right)$
cannot be attained at the interior point.
\end{proof}

\begin{lemma}\label{fundamentalJ} There exists $\delta$ such that for all $|a|<\delta$
$J_a^+\cap \Omega$ is a fundamental domain for $J_a^+\backslash
J_a$.
\end{lemma}

\begin{proof}\cite{HOVII} There it is prove for H\'{e}non mappings that
are perturbations of hyperbolic polynomials with connected Julia
set. In the connected case the proof is the same.
\end{proof}

\begin{lemma}\label{lem2} Suppose $|a|<\delta$ and $C_a$ is a component of the critical locus. Then
there exists an iterate of  $f_a^n(C_a)$  that intersects
$\Omega$.
\end{lemma}

\begin{proof}
Lemma \ref{boundary} states that there exists $z\in (J_a^+ \cup
J_a^-)\cap (\partial C_a)$. Suppose $z\in J_a^-$. Take a sequence
of points $z_n\in C_a$, $z_n\to z$.

$G_a^-(z_n)\to 0$ as $n\to \infty$; $G_a^+$ is bounded.

For every $n$ there exists $k_n$ such that
$1<G_a^-(f_a^{-k_n}(z_n))\leq 2$. Then $k_n\to \infty$.
$$G_a^+(f_a^{-k_n}(z_n))\to 0.$$

\noindent Taking a subsequence one may assume $f^{-k_n}(z_n)\to
z$. $z\in J_a^+\backslash J_a$. Since by Lemma \ref{fundamentalJ}
$f_a^m(z)\in J_a^+\cap \Omega$, there exists an iterate of $C_a$
that intersects $\Omega$.
\end{proof}

Below we will prove that ${\cal C}_a\cap \left(\Omega\cup
\hat{\Omega}_1\right)$ is a fundamental domain of the critical
locus.

The next lemma is a variation of the Classical Poincar\'{e}'s
Polyhedron Theorem on the fundamental domain under the group
action (\cite{Maskit}, IV.H).

\begin{lemma}\label{abstract_lemma} Let $C$ be a Riemann surface. Let $f:C\to C$ be an automorphism.
Let $D\subset C$ be an open domain. Suppose
\begin{enumerate}
\item $f_a^n(D)\cap D=\emptyset$;
\item $\partial D\subset C$ consists of countably many smooth
curves $\gamma_i$;
\item $\gamma_i$ are being paired by the map $f$;
\item for any sequence $\{z_n\}$, $z_n\in D$, $f_a^n(z_n)$
does not have an accumulation point in $S$
\end{enumerate}
Then $D$ is a fundamental domain of $S$ for the map $f_a$.
\end{lemma}

\begin{proof} Denote by $S$ a Riemann surface obtained by gluing
the images $f^n(D)$ to $D$. The natural map $i:S\to C$ is
injective by (1). Let us prove that it is proper. If it is not,
then there exists a sequence of points $z_n\in S$ so that $z_n\to
\partial S$, $i(z_n)\to z\in C$. Take $z_n'\in D$ so that
$z_n=f^{k_n}(z_n')$. If there are infinitely many same $k_n$'s.
Then the sequence $\{z'_n\}$ has an accumulation point in $D$.
That contradicts (2) and (3).

Therefore, taking a subsequence one can assume $\{k_n\}$ increase.
This contradicts (4).
\end{proof}

\begin{lemma} The images of $\left(W\backslash\{|p(y)-x|\geq
a|\alpha|\}\right)\cup\{|y|<\epsilon\}$ under the map $f_a$ are
disjoint.
\end{lemma}

\begin{proof}
\begin{enumerate}
\item $$f_a^n(W\backslash \{|p(y)-x|\leq |a|\alpha\})\cap \left[W\backslash \{|p(y)-x|\leq
a|\alpha|\}\right]=\emptyset.$$

$f_a(W)\subset W\cup V_+$, $f_a(V_+)\subset V_+$, therefore,
$\left(f_a^n(W)\cap W\right)\subset \left(f_a(W)\cap
W\right)\subset \{|p(y)-x|\leq |a|\alpha\}$

\item $$f_a^n(W)\cap \{|y|<\epsilon\}\cap \{|x|\geq \alpha\}=\emptyset;$$

Let $(x,y)\in W$. Suppose $f_a^n(x,y)\in V_+$ and $n$ is the
smallest number such that this happens. Then $|p(x_n)-y_n|\leq
|a|\alpha$. Thus, $|y_n|>\epsilon$. For points in $V_+$,
$|y_{n+1}|> |y_n|$.

\item $$f_a^n(\{|y|<\epsilon\}\cap \{|x|>\alpha\})\cap \{|y|<\epsilon\}\cap \{|x|>\alpha\}=\emptyset.$$

Suppose $(x,y)\in \{|y|<\epsilon\}\cap \{|x|>\alpha\}$. Since
$|y|<\epsilon$, $f_a(x,y)$ belongs to $|a|\epsilon$-neighborhood
of parabola. Since $(x,y)\in V_+$, $f_a(x,y)\in V_+$. Therefore,
$|y_1|>\epsilon$. $|y_{n+1}|>|y_n|>\epsilon$, since $(x_n,y_n)\in
V_+$.

\item $$f_a^n(\{|y|<\epsilon\}\cap \{|x|>\alpha\})\cap
W=\emptyset.$$

This is true, since $\{|y|<\epsilon\}\cap \{|x|>\alpha\}\subset
V_+$, and $f_a(V_+)\subset V_+$.
\end{enumerate}

\end{proof}

\begin{lemma} There exists $\delta$ such that for all $|a|<\delta$ the critical locus in $\Omega\cup \hat{\Omega}_1$
forms a fundamental domain of ${\cal C}_a$.
\end{lemma}

\begin{proof}

Denote by $D={\cal C}_a\cap \left(\Omega\cup
\hat{\Omega}_1\right)$. Let us check that the conditions of the
Lemma \ref{abstract_lemma} are satisfied. Since $\Omega\cup
\hat{\Omega}\subset \left(W\backslash \{|p(y)-x|\leq
|a|\alpha\}\right)\cup \{|y|\leq \epsilon\}$. All the forward
images of $D$ are disjoint from it. Therefore, (1) is satisfied.

The boundary of $D$ in each $\Omega_a^{\xi_k}$ consists of
vertical circles: $|y|=\alpha$ and $|p(y)-x|=|a|\alpha$. The
circles $|p(y)-x|=|a|\alpha$ are parametrized by $\xi'_{k+1}$,
$\xi''_{k+1}$,($p(\xi'_{k+1})=p(\xi''_{k+1})=\xi_k$), which stands
for the approximate value of $y$ on the circle.

There is also one horizontal circle $|p(y)-x|=|a|\alpha$ on a
perturbation of $y=0$.

$f_a$ maps $|y|=\alpha$ in $\Omega_a^{\xi_k}$ to
$|p(y)-x|=|a|\alpha$ in $\Omega_a^{p(\xi_k)}$, parametrized by
$\xi_k$

$f_a$ maps $|y|=\alpha$ on perturbation of $x=0$ to a horizontal
circle $|p(y)-x|=|a|\alpha$.

Therefore, boundary components of $D$ are being paired by $f_a$.
And condition (2) and (3) of Lemma \ref{abstract_lemma} are
satisfied.

Suppose there exists a sequence of $z_n\in D$ so that
$f_a^n(z_n)\to\partial S,$ $z_n\to z_*\in C$.

If $\{z_n\}$ has an accumulation point $z$ in $D$. Then
$f_a^{n}(z)\to z^*$. Which is impossible, since $f_a^{n}(z)\to
\infty$ in $\mathbb C^2$.

If $\{z_n\}$ does not have an accumulation point $z$ in $D$. Then
$z_n$ accumulate to $z\in J_a^+$. Therefore, $f_a^{k_n}(z')\to
z\in J_a$. That contradicts to $z\in C$.

So condition (4) is satisfied as well. Therefore, $D$ is a
fundamental domain of the critical locus ${\cal C}_a$.
\end{proof}

\begin{proof}[\it Proof of theorem \ref{theorem_main}.]
To obtain a description in terms of truncated spheres we do a
dynamical regluing. We fix some small $\epsilon$ and cut the
fundamental domain of ${\cal C}_a$ along the hypersurface
$|y|=\epsilon$. We call the connected component of the
perturbation of $y=0$ main component. The rest of components we
call handles: $H_{\xi_k}$ is a component in $\Omega_{\xi_k}$

The boundary of $H_{\xi_k}$ consists of four circles:

$|y|=\alpha,$ $|y|=\epsilon$ and two connected components of
$|p(y)-x|=|a|\alpha$, parametrized as previously by $\xi'_{k+1}$,
$\xi''_{k+1}$, where $p(\xi'_{k+1})=p(\xi''_{k+1})=\xi_k$.

We glue $H_{\xi_k}$ to the main component by the map $f_a^k$.
Under this procedure the boundary $|y|=\alpha$ of $H_{\xi_k}$ is
being glued to $|p(y)-x|=|a|\alpha$-boundary of $H_{p(\xi_k)}$,
parametrized by $\xi_k$. By ``generalized uniformization theorem",
it can be straighten to be a sphere.

The fundamental domain of the critical locus ${\cal C}_a$,
obtained after regluing, is a truncated sphere.

The preimages of $0$ under the map $p^k$ are parametrized by
$k$-strings of $0$ and $1$'s.

Let $\alpha_k$ be a $k$-string that parametrizes $\xi_k$.
$V_{\alpha_k}$ is the interior of $|y|=\epsilon$ obtained by
cutting $H_{\xi_k}$ from the main component. $U_{\alpha_k}$ is the
interior of $f_a^k (|y|=\epsilon)$ on $f_a^k(H_{\xi_k})$. The rest
of the boundary corresponds to $G_a^+=0$ and $G_a^-=0$. Therefore,
it is parametrized by two Cantor sets $\Sigma, \Omega$.

We show these are true Cantor sets by moduli counting.

\begin{lemma} Let $U_1\supset U_2\supset \dots U_n\supset \dots$
be a sequence of open domains, such that $U_i\backslash U_{i-1}$
is an annulus with moduli $M_i\geq M$. Then $\bigcap \bar{U}_i$ is
a point.
\end{lemma}

Take a point $\sigma\in \Sigma$. Let $M_1$ be a modulus of the
annulus $\{r-\lambda\leq G_a^+\leq r\}$ on $y=0$. Then $M_1$ is a
modulus of the annulus $\{\frac{r-\lambda}{2^n}\leq G_a^+\leq
\frac{r}{2^n}\}$. The main component project one-to-one to each of
this annulus. There is a sequence of connected component
$\{\frac{r-\lambda}{2^n}\leq G_a^+\leq \frac{r}{2^n}\}$ that bound
a hole parametrized by $\sigma$. Therefore, $\sigma$ is a point.
$\Sigma$ is a Cantor set.

Fix some small $\epsilon$. Let $M_1$ be a modulus of the annulus
$\{\alpha-\epsilon\leq|y|\leq\alpha\}$. All handles $H_{\xi_k}$
project one-to-one to this annulus. So by the same argument we get
that $\Omega$ is a Cantor set.
\end{proof}

\newpage

\section{List of standard notations}
We provide a list of notations here as a reference.\medskip

\begin{center}
\begin{tabular}{|l|l|p{11cm}|}\hline

{\bf Notation} & {\bf Section} & {\bf Meaning}\\ \hline

$f_a$ & \ref{introduction}& H\'{e}non mapping under
consideration\\ \hline

$a$ & \ref{introduction} & Jacobian of H\'{e}non mapping under
consideration\\ \hline

$p(x)$ & \ref{introduction} & polynomial used in the definition of
H\'{e}non mapping \\ \hline

$U_a^+, U_a^-$ & \ref{introduction} & set of points whose orbits
under forward (backward) iteration of $f_a$ escape to infinity\\
\hline

$K_a^+, K_a^-$ & \ref{introduction} & set of points whose orbits
under forward (backward) iteration of $f_a$ remain bounded\\
\hline

$J_a^+, J_a^-$ & \ref{introduction} & boundaries of $K_a^+,
K_a^-$\\
\hline

$J_a$ & \ref{introduction} & $J_a^+\cap J_a^-$\\
\hline

$G_a^+, G_a^-$ & \ref{introduction} & pluriharmonic functions that
measure the rate of escape to infinity under forward (backward)
iterates of $f_a$\\
\hline

$\alpha$ & \ref{bottcher_coordinates} & parameter used in the
definition of $V_+$, $V_-$\\
\hline

$V_+, V_-$ & \ref{bottcher_coordinates} & $\{|x|> \alpha, |x|>
|y|\}$, $\{|y|> \alpha, |y|> |x|\}$ -- regions which
describe the large scale behavior of the H\'{e}non map\\
\hline

$W$ & \ref{bottcher_coordinates} & $\{|x|\leq\alpha, |y|\leq\alpha\}$\\
\hline

$D_R$ & \ref{bottcher_coordinates} &  the disk of radius $R$ in
the parameter space, used to define $V_+$ and $V_-$\\
\hline

$\phi_{a,+}, \phi_{a,-}$ & \ref{bottcher_coordinates} &
holomorphic functions that semiconjugate dynamics in $V_+, V_-$ to
$z\to z^2$,
$z\to z^2/a$\\
\hline

$s_k^+, s_k^-$ & \ref{bottcher_coordinates} & auxillary functions,
used to study $\phi_{a,+}$ and $\phi_{a,-}$\\
\hline

$C(p)$ & \ref{bottcher_coordinates} & the curve $y=p(x)$, this is $J_0^-$\\
\hline

$G_p$ & \ref{strategy} & Green function for the map
$x\to p(x)$\\
\hline

$b_p$ & \ref{bottcher_coordinates} & B\"{o}ttcher coordinate for the map $x\to p(x)$\\
\hline

${\cal F}_{a}^+, {\cal F}_a^-$ & \ref{introduction} & foliations of $U_a^+$, $U_a^-$\\
\hline

${\cal C}_a$ & \ref{introduction}& the critical locus, the set of
tangencies between foliations ${\cal F}_a^+$ and ${\cal F}_a^-$\\
\hline


$(x_n,y_n)$ & \ref{bottcher_coordinates} & $(x_n,y_n)=f_a(x,y)$\\
\hline

${\cal D}_{k,+}$ & \ref{bottcher_coordinates} & the domain of
definition of $\phi_{a,+}^k$\\
\hline

${\cal D}_{k,-}$ & \ref{bottcher_coordinates} & the domain of
definition of $\phi_{a,-}^k$\\

\hline

$u$ && $u=p(y)-x$, measures the distance from a point $(x,y)$ to $C(p)$\\
\hline

$\Omega_a$ & \ref{strategy} & $\{G_a^+\leq r\}\cap\{|y|\leq
\alpha\}\cap \{|p(y)-x|<\alpha|a|\}$, the domain that does not
intersect with its images under $f_a$ and $f^{-1}_a$ \\
\hline

$\hat{\Omega}_1$ & \ref{final} & $\{(x,y)\in \mathbb C\mathbb
P^2|\quad |y|<\epsilon, |x|\geq \alpha\}$ \\\hline
\end{tabular}
\end{center}


\newpage

\begin{thebibliography}{99}

\bibitem[AGV87]{AGVII} Arnol'd, V.I.; Gusein-Zade, S.M.; Varchenko, A.N.
Singularities of differentialbe maps. Vol.II. Monodromy and
assymptotics of integrals, under the editorship of V.I. Arnold,
translated from Russian by Hugh Porteous. Monographs in
Mathematics, 83. Birhauser Boston, Inc., Boston, MA, 1987

\bibitem[BLS93]{BLS} Eric Bedford, Mikhail Lyubich, and John
Smillie. Polynomial diffeomorphisms of $\mathbb C^2.$ IV. The
measure of maximal entropy and laminar currents. Invent. Math.,
112(1):77-125, 1993.

\bibitem[BS91a]{BS1} Eric Bedford and John Smillie. Polynomial
diffeomorphisms of $\mathbb C^2:$ currents, equilibrium measure
and hyperbolicity. Invent. Math., 103(1):69-99, 1991

\bibitem[BS91b]{BS2} Eric Bedford and John Smillie. Polynomial
diffeomorphisms of $\mathbb C^2.$ II, Stable manifolds and
recurrence. J. Amer. Math. Soc., 4(4):657-679, 1991.

\bibitem[BS92]{BS3} Eric Bedford and John Smillie. Polynomial
diffeomorphisms of $\mathbb C^2.$ III. Ergodicity, exponents and
entropy of the equlibrium measure. Math. Ann., 294(3): 395-420,
1992

\bibitem[BS98a]{BS4} Eric Bedford and John Smillie. Polynomial
diffeomorphism of $\mathbb C^2$. V. Critical points and Lyapunov
exponents. J. Geom. Anal., 8(3):349-383, 1998.

\bibitem[BS98b]{BS5} Eric Bedford and John Smillie. Polynomial
diffeomorphisms of $\mathbb C^2.$ VI. Connectivity of J. Ann. of
Math. (2), 148(2): 695-735, 1998

\bibitem[BS99]{BS6} Eric Bedford and John Smillie. Polynomial
diffeomorphisms of $\mathbb C^2$. VII. Hyperbolicity and external
rays. Ann. Sci. Ecole Norm. Sup.(4), 32(4):455-497, 1999

\bibitem[FM89]{FM} Shmuel Friedland and John Milnor. Dynamical
properties of plane polynomial automorphisms. Ergodic Theory
Dynam. Systems, 9(1):67-99, 1989.

\bibitem[FS92]{FS} John Erik Fornaess and Nessim Sibony. Complex
H\'{e}non mappings in $\mathbb C^2$ and Fatou-Bieberbach domains.
Duke Mathematical Journal, 65(2):345-380, 1992

\bibitem[HOV94]{HOV1} John H. Hubbard and Ralph W. Oberste-Vorth. H\'{e}non mappings
in the complex domain. I. The global topology of dynamical space. Inst. Hautes \'{E}tudes Sci.
Publ. Math., (79):5-46, 1994.

\bibitem[HOV95]{HOVII} John H. Hubbard and Ralph W. Oberste-Vorth. H\'{e}non
mappings in the complex domain. II. Projective and inductive
limits of polynomials. Real and Complex dynamical systems
(Hillerod, 1993), volume 464 of NATO Adv. Sci. Inst. Ser. C Math.
Phys. Sci., pp. 89-132, Kluwer Acad. Publ., Dordrecht, 1995.

\bibitem[M88]{Maskit} B. Maskit, Kleinian Groups. Springer-Verlag,
1988

\bibitem[LR]{LR} Mikhail Lyubich, John Robertson. The critical
locus and rigidity of foliations of complex H\'{e}non maps.
Unpublished.

\end{thebibliography}
\end{document}